\documentclass[a4paper,11pt,
  leqno
 ]{article}

 \usepackage{comment}
 \usepackage[utf8]{inputenc}
 \usepackage[T1]{fontenc}
 \usepackage{amsthm}
 \usepackage{amssymb}
 \usepackage{mathrsfs}
 \usepackage{amsmath}
 \usepackage{amsfonts}
 \usepackage{mathtools}
 \usepackage{graphicx}
 \usepackage{float}
\usepackage{dsfont}
\usepackage{enumerate}
\usepackage{enumitem}
\usepackage[a4paper]{geometry}
\usepackage{csquotes}
\usepackage{IEEEtrantools}
\usepackage{appendix}
\usepackage{constants}
\usepackage[obeyDraft]{todonotes}
\usepackage{tikz-cd}
\usepackage{enumerate}
\usepackage{mathabx}
\usepackage{nicefrac}
\usepackage{hyperref}
\hypersetup{linktocpage,
  colorlinks   = true, 
  urlcolor     = blue, 
  linkcolor    = blue,
  citecolor   = blue 
}

\newtheorem{The}{Theorem}[section]

\newtheorem{lem}[The]{Lemma}

\newtheorem{prop}[The]{Proposition}
\newtheorem{Cor}[The]{Corollary}

\theoremstyle{definition}

\theoremstyle{remark}

\newtheorem{Rk}[The]{Remark}

\numberwithin{equation}{section}

\newcommand{\rrvert}{\vert}
\newcommand{\llvert}{\vert}

\usepackage{caption}
\title{
\normalsize
\textbf{{
ARM EXPONENT 
FOR
THE GAUSSIAN FREE FIELD ON METRIC GRAPHS IN 
INTERMEDIATE
DIMENSIONS}}}
\author{}
\date{}
\makeatletter
\renewcommand{\phi}{\varphi}

\renewcommand{\epsilon}{\varepsilon}

\newconstantfamily{c}{symbol=c}

\usepackage{color}
\definecolor{Red}{rgb}{1,0,0}
\definecolor{Blue}{rgb}{0,0,1}
\definecolor{Olive}{rgb}{0.41,0.55,0.13}
\definecolor{Yarok}{rgb}{0,0.5,0}
\definecolor{Green}{rgb}{0,1,0}
\definecolor{MGreen}{rgb}{0,0.8,0}
\definecolor{DGreen}{rgb}{0,0.55,0}
\definecolor{Yellow}{rgb}{1,1,0}
\definecolor{Cyan}{rgb}{0,1,1}
\definecolor{Magenta}{rgb}{1,0,1}
\definecolor{Orange}{rgb}{1,.5,0}
\definecolor{Violet}{rgb}{.5,0,.5}
\definecolor{Purple}{rgb}{.75,0,.25}
\definecolor{Brown}{rgb}{.75,.5,.25}
\definecolor{Grey}{rgb}{.7,.7,.7}
\definecolor{Black}{rgb}{0,0,0}

\usepackage{titletoc}

\dottedcontents{section}[4em]{}{2.9em}{0.7pc}
\dottedcontents{subsection}[0em]{}{3.3em}{1pc}

\usepackage{multicol}
\usepackage{parcolumns}

\usepackage{tabu}
\usepackage{caption}
\begin{document}
\thispagestyle{empty}
\maketitle
\vspace{0.1cm}
\begin{center}
\vspace{-1.9cm}
Alexander Drewitz$^1$, Alexis Pr\'evost$^2$ and Pierre-Fran\c cois Rodriguez$^{3,4}$

\end{center}
\vspace{0.1cm}
\begin{abstract}
\centering
\begin{minipage}{0.85\textwidth}
We investigate the bond percolation model on transient weighted graphs
${G}$ induced by the excursion sets of the Gaussian free field on the corresponding
metric graph. We assume that balls in ${G}$ have polynomial volume growth
with growth exponent $\alpha $ and that the Green's function for the random
walk on ${G}$ exhibits a power law decay with exponent $\nu $, in the regime
$1\leq \nu \leq \frac{\alpha}{2}$. In particular, this includes the cases
of ${G}=\mathbb{Z}^{3}$ for which $\nu =1$, and
${G}= \mathbb{Z}^{4}$ for which $\nu =\frac{\alpha}{2}=2$. For all such
graphs, we determine the leading-order asymptotic behavior for the critical
one-arm probability, which we prove decays with distance $R$, like
$R^{-\frac{\nu}{2}+o(1)}$. Our results are, in fact, more precise and yield
logarithmic corrections when $\nu >1$ as well as corrections of order
$\log \log R$ when $\nu =1$. We further obtain very sharp upper bounds
on truncated two-point functions close to criticality, which are new when
$\nu >1$ and essentially optimal when $\nu =1$. This extends previous results
from \cite{DrePreRod5, DiWi}.
\end{minipage}
\end{abstract}

\vspace{5.5cm}
\begin{flushleft}

\noindent\rule{5cm}{0.4pt} \hfill October 2025 \\
\bigskip
\begin{multicols}{2}

$^1$Universit\"at zu K\"oln\\
Department Mathematik/Informatik \\
Weyertal 86--90 \\
50931 K\"oln, Germany. \\
\url{adrewitz@uni-koeln.de}\\[2em]

$^2$Institute for Applied Mathematics\\
University of Bonn\\
Endenicher Allee 60\\
D-53115 Bonn, Germany.
\\\url{prevost@iam.uni-bonn.de}\\[2em]

\columnbreak
\thispagestyle{empty}
\bigskip
\medskip
\hfill$^3$Imperial College London\\
\hfill Department of Mathematics\\
\hfill 180 Queen's Gate\\
\hfill  SW7 2AZ London, UK\\
\hfill \\[2em]
\hfill$^3$Center for Mathematical Sciences\\
\hfill University of Cambridge\\
\hfill Wilberforce Road\\
\hfill  CB3 0WB Cambridge, UK \\
\hfill \url{pfr26@cam.ac.uk} 
\end{multicols}
\end{flushleft}

\newpage

\section{Introduction}
\label{sec:intro}

Percolation models exhibit intriguing behavior at and near their critical
point, which is notoriously difficult to describe rigorously. Among many
quantities of interest, one which plays a central role in this article
is the ``one-arm'' probability to connect a point to distance
$R\geq 1$. For Bernoulli site percolation on the two-dimensional triangular
lattice, this observable was famously shown \cite{MR1887622,Smi01} to decay
polynomially in $R$ at criticality as $R^{-\frac{5}{48} +o(1)}$. On
$\mathbb{Z}^{d}$ for sufficiently high dimension $d$ (namely,
$d\geq 11$), the decay is known to be of order $R^{-2}$ independently of
$d$ \cite{KozNac11,10.1214/17-EJP56,HarSla90,Bar-Aiz91}, a manifestation
of mean-field behavior; see also \cite{HarHofSla03} concerning spread-out
models for which this regime has been proven to extend to all $d>6$. In
intermediate dimensions $d=3,4,\dots $, however, the mere decay of the
critical one-arm probability in itself is an outstanding open problem.

Recently, a bond percolation model with long-range correlations involving
the Gaussian free field, which belongs to a different universality class
than Bernoulli percolation, has led to significant advances on questions
of the above type, notably in the challenging intermediate dimensions
\cite{DiWi,DrePreRod5,DrePreRod3,MR3502602,werner2020clusters,cai2023onearm}.
We summarize these below; see also
\cite{Lupu22,jego2023crossing,MR4091511,zbMATH07047497,DrElPrTyVi} for
related results, including in dimension two. One focus of the present article
is the one-arm decay of this model at criticality in intermediate dimensions.
Another central quantity of interest is the truncated two-point function
near criticality for which we derive very sharp upper bounds. Together,
these results have important ramifications concerning the behavior of critical
and near-critical cluster volumes. Moreover, in the special case of
$\mathbb{Z}^{3}$, they refine rather drastically recent results of
\cite{GRS21}, which concern a different (but related) model. Indeed, the
quantitative two-point estimates we derive witness rotational invariance
\textit{at} the correlation length scale.

Our results hold under certain mild conditions on the base graph, which
originate in \cite{MR1853353} (see also \cite{DrePreRod2} in the context
of percolation for the Gaussian free field) and are not specific to Euclidean
lattices. We consider $\mathcal{G}=(G, \lambda )$ a transient weighted
graph, connected and locally finite. We assume controlled weights, that
is,
%
\begin{align}
& \label{eq:ellipticity} \tag{$p_{0}$} \lambda _{x,y}/\lambda
_{x}\geq c\quad \text{for all }x, y\in{G} \text{ s.t. $\lambda _{x,y}>0$,}
\end{align}
for some $c>0$, where
$\lambda _{x} \stackrel{\mathrm{def.}}{=} \sum_{y \sim x} \lambda _{x,y}$
for $x \in G$. We further impose two natural conditions on the growth of
balls and the decay of the Green's function for the random walk on
$\mathcal{G}$. Namely, there exist a positive exponent $\alpha $ and
$ c,C \in (0 ,\infty )$ such that the volume growth condition
\begin{equation}
\label{eq:intro_sizeball} \tag{$V_{\alpha}$} cr^{\alpha}\leq \lambda
\bigl(B(x,r) \bigr)\leq Cr^{\alpha}\quad \text{ for all }x
\in{G}\text{ and }r\geq 1
\end{equation}
is satisfied, where $B(x,r)$ refers to the discrete ball of radius
$r$ around $x \in G$ in a given metric $d$ on $G$ and $\lambda $ denotes
the measure induced via the point masses $\lambda _{x}$ for
$x \in G$. Moreover, there exist an exponent $\nu >0$ and constants
$ c,C \in (0 ,\infty )$ such that
\begin{equation}
\label{eq:intro_Green} \tag{$G_{\nu}$} %
\begin{aligned} c\leq
g(x,x) \leq C \quad \text{and}\quad c d(x,y)^{-\nu}\leq g(x,y)
\leq C d(x,y)^{-\nu} \quad \text{for all }x \neq y\in{G},
\end{aligned} %
\end{equation}
where
$g(x,y) \stackrel{\mathrm{def.}}{=} \sum_{n=0}^{\infty }P_{x}(Z_{n}=y)
\frac{1}{\lambda _{y}}$, $x,y \in G$, denotes Green's function for the
random walk $(Z_{n})$ in conductances $\lambda $ on $G$ for which
$P_{x}(Z_{n}=0)=x$; see Section~\ref{sec:pre} for further details. As explained
around \cite{DrePreRod5}, (1.18), in case $d=d_{\mathrm{gr}}$, where
$d_{\mathrm{gr}}$ is the graph distance on $\mathcal{G}$, these conditions
imply that $0<\nu \leq \alpha -2$. We assume from now on that these bounds
on $\nu $ are satisfied. An emblematic example of graphs satisfying these
conditions are the Euclidean lattices $G= \mathbb{Z}^{\alpha}$ for integer
$\alpha \geq 3$ endowed with weights
$\lambda _{x,y}= \tfrac1{2d}\{ |x-y|=1 \}$, $x,y \in G$; here $|\cdot |$ denotes
the Euclidean distance, and this setting fits the above setup with
$\nu = \alpha -2$. We always tacitly assume these choices of weights when referring to $G=\mathbb{Z}^{\alpha}$ in the sequel. Whereas $\nu $ thus ranges over the integer values
$ \{1,2,\dots \}$ on the lattice, intermediate values are also attainable. Examples fulfilling the above conditions include
%
\begin{equation}
\label{intro:Gex1} %
\begin{aligned} G_{1} &=
\mathbb{Z}^{d}\quad \text{with } d \geq 3,
\\
G_{2} &= G' \times \mathbb{Z}\quad
\text{with $G'$ the discrete skeleton of the Sierpinski gasket},
\\
G_{3} &= \text{the standard $d$-dimensional graphical Sierpinski carpet for
$d\geq 3$,}
\\
G_{4}&= %
\begin{aligned} &\text{a Cayley graph of a finitely generated group
$\Gamma = \langle S \rangle $ with $S=S^{-1}$}
\\[-3pt]
&\text{having polynomial volume growth of order $\alpha >2$}; \end{aligned} \end{aligned} %
\end{equation}
see \cite{MR2076770} and also \cite{DrePreRod2} for further details. Although
the examples mentioned above often entail additional symmetry, no assumptions
beyond \eqref{eq:intro_sizeball} and \eqref{eq:intro_Green} (along with
controlled weights) are required, in particular, no transitivity assumption
needs to be imposed.

Let ${\varphi}$ denote the Gaussian free field on ${\mathcal{G}}$, that
is, the centered Gaussian process with covariance function given by
$g$, and write $\mathbb{P}$ for the underlying probability measure. The
excursion sets $\{x \in G: \varphi _{x} \geq a\}$ for varying parameter
$a \in \mathbb{R}$ lead to a natural \textit{site} percolation model on
$G$ and have been extensively studied; see,~for example,~\cite{MR914444,MR3053773,DrRo-13,MR3417515,DrePreRod,GRS21,MR4568695,muirhead2022percolation}.
A variation with improved integrability properties is obtained by considering
instead the \textit{bond} percolation model, which is obtained by retaining
each edge $\{x,y\}$ in $G$ independently with probability
%
\begin{equation}
\label{eq:bond-proba} 1-\exp \bigl\{-2\lambda _{x,y}(\varphi
_{x}-a)_{+} (\varphi _{y}-a)_{+}
\bigr\}, \quad a \in \mathbb{R}, 
\end{equation}
conditionally on $\varphi $, where $t_{+} =\max \{t,0\}$ for
$t \in \mathbb{R}$. We assume $\mathbb{P}$ to be suitably extended to carry
this additional randomness. An edge is called open if it has been retained.
Let $E^{\geq a}$ denote the set of open edges of $\mathcal{G}$ obtained
in this way. In view of \eqref{eq:bond-proba}, by a straightforward coupling
involving uniform random variables, one sees that $E^{\geq a}$ is decreasing
in $a$, thus rendering the phase transition for the percolation problem
$(E^{\geq a})_{a\in \mathbb{R}}$ well defined. What lurks behind the choice
of the bond disorder \eqref{eq:bond-proba} is an extension of
$\varphi $ to the metric graph $\widetilde{\mathcal{G}} \supset G$, comprising
$G$ along with one-dimensional cables joining neighboring vertices in
$G$. The percolation problem for $(E^{\geq a})_{a\in \mathbb{R}}$ is equivalent
to the percolation problem for excursion sets of the continuous extension
of $\varphi $ to $\widetilde{\mathcal{G}}$; see Section~\ref{sec:pre} for
details.

As a consequence of \cite{DrePreRod3}, Theorem~1.1(1), which is applicable
due to Lemma~3.4(2) therein and our assumption \eqref{eq:intro_Green},
one knows the following:
\begin{itemize}
\item[{(i)}] The percolation model $(E^{\geq a})_{a\in \mathbb{R}}$ has a nontrivial
phase transition on any graph $\mathcal{G}$ satisfying the above requirements.
\item[{(ii)}] The critical height is $a=0$ regardless of the particular choice
of $\mathcal{G}$.
\item[{(iii)}] The order parameter of the transition (i.e.,~the probability of a
point to belong to an infinite cluster in $E^{\geq a}$) is a continuous
function of $a$.
\end{itemize}

Our first result concerns the critical ``one-arm'' events
\begin{equation*}
\bigl\{ x \leftrightarrow B(x,R)^{\mathsf{c}} \bigr\},
\end{equation*}
which refer to the existence of a path in $E^{\geq 0}$ connecting
$0$ and $B(x,R)^{\mathsf{c}}=G\setminus B(x,R)$. To date, the best available
bounds on the critical one-arm probability
$\mathbb{P}(x \leftrightarrow B(x,R)^{\mathsf{c}})$ are as follows. By
\cite{DrePreRod5}, (1.22)--(1.23), one knows that, for all $\nu >0$,
$x\in{G}$, and $R \geq 1$, and some $c,C \in (0,\infty )$ (see the end
of this Introduction regarding our convention with constants),
%
\begin{align}
\label{eq:1-arm_LB}
\mathbb{P} \bigl(x \leftrightarrow B(x,R)^{\mathsf{c}} \bigr) \geq cR^{-
\frac{\nu }{2}},
\end{align}
and
%
\begin{align}
&
\label{eq:1-arm_UB-old}
\mathbb{P} \bigl(x \leftrightarrow B(x,R)^{\mathsf{c}} \bigr) \leq C
\begin{cases}
R^{-\frac {\nu }{2}} & \text{if } \nu <1,
\\
\biggl(\frac{R}{\log R} \biggr)^{-\frac{1}{2}} & \text{if }
\nu =1,
\\
R^{-\frac{1}{2}} & \text{if } \nu >1.
\end{cases}
\end{align}
The bounds \eqref{eq:1-arm_LB}--\eqref{eq:1-arm_UB-old} were first derived
in \cite{DiWi} for $G=\mathbb{Z}^{\alpha}$, with
$\alpha =\nu +2 (\geq 3)$, via different methods. Later on, in the general
graph setting of \cite{DrePreRod5}, they were seen to follow immediately
from a comparison between the radius and capacity observables for the cluster
of a point, using that the latter is integrable; see
\cite{DrePreRod5}, Corollary~1.3. In fact, the proof of
\eqref{eq:1-arm_LB}--\eqref{eq:1-arm_UB-old} does not even require the
assumption \eqref{eq:intro_sizeball}. The bounds \eqref{eq:1-arm_LB}--\eqref{eq:1-arm_UB-old}
thus express little more than the fact that the cluster of $x$ reaching
distance $R$ can be anything from a line to the full box. When
$\nu <1$, the two are indistinguishable to leading order in $R$ when measured
in terms of their capacity, but discrepancies start to arise when
$\nu \geq 1$, which corresponds to dimension three and higher on Euclidean
lattices. In the mean-field regime of $\mathbb{Z}^{\alpha}$ for any integer
$\alpha >6$, it was recently proved in \cite{cai2023onearm} that
$\mathbb{P} (x \leftrightarrow B(x,R)^{\mathsf{c}} )\asymp R^{-2}$;
see also \cite{werner2020clusters} for more results valid in high dimensions.

The mismatch between \eqref{eq:1-arm_LB} and \eqref{eq:1-arm_UB-old} for
$\nu \geq 1$ warrants further investigation, in particular, in the regime
of ``intermediate'' dimensions, for example, for integer $\alpha $ with
$3 \leq \alpha \leq 6$ on $\mathbb{Z}^{\alpha}$, which is the object of
our first main result. From here on, we impose the standing assumption
%
\begin{equation}
\label{eq:restr-params} 1\leq \nu \leq \frac{\alpha}{2}. 
\end{equation}
In particular, \eqref{eq:restr-params} includes the cases
$G=\mathbb{Z}^{\alpha}$, $\alpha \in \{3,4\}$, for which
$\nu =\alpha -2 (\leq \frac{\alpha}{2})$, but also, for example,~certain
fractal graphs from \eqref{intro:Gex1}, as will be mentioned below. Our first result is an improved
upper bound on the one-arm probability at criticality.

\begin{The}%
\label{T:1arm-crit}
There exists $C\in (0,\infty )$ such that, for all $x\in{G}$ and
$R \geq 3$,
%
\begin{align}
\label{eq:1-arm-UB-crit}
\mathbb{P}  \bigl(x \leftrightarrow B(x,R)^{\mathsf{c}}  \bigr) \leq Cq(R) R^{-
\frac{\nu }{2}},
\end{align}
where
%
\begin{equation}
\label{eq:const-UB-1arm-crit} q(R)= %
\begin{cases} \log \log R &
\text{if }\nu =1,
\\
(\log R)^{\frac{1}{2}(\nu -1)}(\log \log R)^{\nu} & \text{if }1<
\nu < \frac{ \alpha}{2},
\\
(\log R)^{\nu}(\log \log R)^{\nu} & \text{if } \nu
=\frac{ \alpha}{2}. \end{cases} %
\end{equation}
In particular, for such $\alpha$ and $\nu$ the critical one-arm exponent satisfies
%
\begin{equation}
\label{eq:def_rho} \frac{1}{\rho} \stackrel{\mathrm{def.}} {=} -\lim
_{R\rightarrow
\infty} \frac{\log \mathbb{P}(x \leftrightarrow B(x,R)^{\mathsf{c}})}{\log R} = \frac{\nu}{2}.
\end{equation}
\end{The}
Thus, in comparison with \eqref{eq:1-arm_LB}--\eqref{eq:1-arm_UB-old},
Theorem~\ref{T:1arm-crit} yields an improvement in the discrepancy between
upper and lower bounds: first, for $\nu =1$, we obtain an improvement of
the upper bound from a factor of $(\log R)^{\frac{1}{2}}$ to
$\log \log R$, which, in particular, applies to the lattice
$\mathbb{Z}^{3}$; see Remark~\ref{rk:simplerproof} for more details. Even
more distinctively, in the regime $1< \nu \leq \frac{ \alpha}{2}$, Theorem~\ref{T:1arm-crit}
implies that the polynomial lower bound \eqref{eq:1-arm_LB} is, in fact,
sharp up to logarithmic factors. For instance, $\mathbb{Z}^{4}$ corresponds
to the case $\nu = \frac{\alpha}{2}=2$; an example with
$1< \nu < \frac{\alpha}{2}$ is the graphical Sierpinski carpet in
$\alpha =4$ dimensions; see \cite{MR1802425} and
\cite{DrePreRod2}, Remark~3.10,2). This sharpness up to logarithmic factors
implies that the limit in \eqref{eq:def_rho} exists and defines a critical
one-arm exponent.

There is, however, no specific reason to believe that the upper bounds
we derive are sharp up to constants. In fact, we prove slightly more than
\eqref{eq:const-UB-1arm-crit}. For instance, in case $\nu =1$, the proof
yields \eqref{eq:1-arm-UB-crit} with
$q(R)= (\log \log R)^{\frac{2}{3}+ \varepsilon}$ for any
$\varepsilon > 0$; see \eqref{eq:end-Th1} for best available bounds in
all cases.

The strategy we follow in order to obtain \eqref{eq:1-arm-UB-crit}--\eqref{eq:const-UB-1arm-crit}
consists of deriving substantial refinements of previous estimates for
the ``thickness'' of the cluster of $x$ when measured in terms of capacity.
This ``capacity thickening'' is due to the underlying presence of random
walk-like objects within this cluster, the capacity of which is particularly
large under the assumption \eqref{eq:restr-params}. We will return to this
below. As mentioned above, on $\mathbb{Z}^{\alpha}$ for (integer)
$\alpha >6$, it was proved in \cite{cai2023onearm} that $\rho =1/2$, and
noting that $2/\nu =1/2$ exactly when $\alpha =6$ (and hence,
$\nu =4$), we conjecture that \eqref{eq:def_rho} remains true on
$\mathbb{Z}^{\alpha}$ for $\alpha \in{\{5,6\}}$.

Theorem~\ref{T:1arm-crit} has further important consequences, which we
now discuss. Indeed, this is because our improved bounds for
$q(\cdot )$ from \eqref{eq:1-arm-UB-crit} and
\eqref{eq:const-UB-1arm-crit} feed into the expressions for the lower bound
on the near-critical one-arm probability and two-point function obtained
in \cite{DrePreRod5}, Proposition~6.1 and (8.3). We detail this in the
latter case. For each $a\in{\mathbb{R}}$ and $x,y\in{G}$, let
%
\begin{equation}
\label{eq:2-point}
\tau _{a}^{\mathrm{tr}}(x,y)= \tau _{a}^{\mathrm{tr}}(y,x)
\stackrel{\text{def.}}{=}\mathbb{P} \bigl(x\leftrightarrow y\text{ in }E^{
\geq a},x\nleftrightarrow \infty \text{ in }E^{\geq a} \bigr)
\end{equation}
denote the truncated two-point function, where ``$x\leftrightarrow y$ in $A$'' means that there is a path of edges in
$A$ from $x$ to $y$. Then under the additional hypothesis that
$d=d_{\mathrm{gr}}$, one obtains the following: there exist constants
$c,\Cl{Ctaulower} \in (0,\infty )$ such that, for all
$a\in{\mathbb{R}}$ with $|a|\leq c$ and all $x,y\in{G}$ with
$d(x,y)\geq \Cr{Ctaulower}|a|^{-2}q(a^{-1})\log (q(a^{-1}))$ when
$\nu =1$, and with
$d(x,y)\geq \Cr{Ctaulower}|a|^{-\frac{2}{\nu}} q(a^{-1}) $ when
$1<\nu < \frac{\alpha}{2}$, one has%
%
\begin{equation}
\label{eq:lowerboundtau} \tau _{a}^{\mathrm{tr}}(x,y)\geq \tau
_{0}^{\mathrm{tr}}(x,y) %
\begin{cases}
\exp \biggl\{-\frac{\Cr{Ctaulower} \llvert a \rrvert ^{2}d(x,y)}
{\log ( \llvert a \rrvert ^{2}d(x,y))} \biggr\}&\text{if }\nu =1,
\\[6pt]
\exp \bigl\{-\Cr{Ctaulower} \llvert a \rrvert ^{\frac{2}{\nu}}d(x,y) \bigl(\log
\bigl( \llvert a \rrvert ^{
\frac{2}{\nu}}d(x,y) \bigr) \bigr)^{\nu -1}
\bigr\}&\text{if }1<\nu < \frac{\alpha}{2}. \end{cases}
%
\end{equation}
We refer to the end of Section~\ref{sec:2point} for the short derivation
of these bounds, which follow from Theorem~\ref{T:1arm-crit} in combination
with \cite{DrePreRod5}, (8.3). Note, however, that, when
$\nu =\frac{\alpha}{2}$, one cannot improve on the bound of
\cite{DrePreRod5}, Theorem~1.7, for reasons explained in
\cite{DrePreRod5}, Remark~8.1,1). Regarding the one-arm probability,
\cite{DrePreRod5}, Proposition~6.1, combined with Theorem~\ref{T:1arm-crit} yields the exact same lower bounds as in
\eqref{eq:lowerboundtau}, but with $\tau _{a}^{\mathrm{tr}}(x,y)$ and
$ \tau _{0}^{\mathrm{tr}}(x,y)$ replaced by the near-critical and critical
one-arm probabilities as well as $d(x,y)$ replaced by the radius
$r$.

Our second main result Theorem~\ref{thm:2pointUB} yields an upper bound
on the truncated two-point function $\tau _{a}^{\mathrm{tr}}(x,y)$ from
\eqref{eq:2-point}, thereby assessing the sharpness of the lower bound
\eqref{eq:lowerboundtau} (up to $\log $ corrections when $\nu >1$) as well
as the resulting lower bound on the correlation length; compare~the discussion
around \eqref{eq:defxi} below. In order to put this into context, we recall
that, to date, the best upper bound was proved in
\cite{DrePreRod5}, Theorem~1.4: there exists $\Cl[c]{ctauupper}>0$ such
that, for all $a\in{\mathbb{R}}$ and $x,y\in{G}$,
%
\begin{equation}
\label{eq:bound2point1} \tau _{a}^{\mathrm{tr}}(x,y)\leq \tau
_{0}^{\mathrm{tr}}(x,y) %
\begin{cases}
\exp \biggl\{-\frac{\Cr{ctauupper} \llvert a \rrvert ^{2}d(x,y)}{\log (d(x,y)) \vee 1} \biggr\}&\text{if }\nu =1,
\\[6pt]
\exp \bigl\{-{\Cr{ctauupper} \llvert a \rrvert ^{2}d(x,y)} \bigr
\}&\text{if }\nu >1; \end{cases} %
\end{equation}
see also \cite{DiWi}, Theorem~4, for a similar result in the case
$G=\mathbb{Z}^{\alpha}$, albeit without prefactor
$\tau _{0}^{\mathrm{tr}}(x,y)$. To be precise,
\eqref{eq:bound2point1} only requires $\mathcal{G}$ to have controlled
weights and to satisfy \eqref{eq:intro_Green}. Let us also remark here
that, in the case $\nu <1$, one already has matching upper and lower bounds;
see \cite{DrePreRod5}, Theorem~1.4, by which
$\log (\tau _{a}^{\mathrm{tr}}(x,y)/\tau _{0}^{\mathrm{tr}}(x,y))$ behaves
like $-(|a|^{\frac{2}{\nu}}d(x,y))^{\nu}$. Our second main result improves
the bound \eqref{eq:bound2point1} in the regime of parameters
$1\leq \nu \leq \frac{\alpha}{2}$ from \eqref{eq:restr-params}.

\begin{The}
\label{thm:2pointUB}
There exist $\Cl[c]{ctwopoint}, \Cl{Ctwopoint} \in (0,\infty )$ such that,
for all $a\in{\mathbb{R}}$ and $x,y\in{G}$,
%
\begin{equation}
\label{eq:bound2point} \tau _{a}^{\mathrm{tr}}(x,y)\leq \Cr{Ctwopoint}
\tau _{0}^{\mathrm{tr}}(x,y) %
\begin{cases} \exp \biggl\{- \frac{\Cr{ctwopoint} \llvert a \rrvert ^{2}d(x,y)}{\log ( \llvert a \rrvert ^{2}d(x,y)) \vee 1} \biggr\}& \text{if }\nu
=1,
\\[10pt]
\exp \bigl\{-{\Cr{ctwopoint} \llvert a \rrvert ^{\frac{2}{\nu}}d(x,y)} \bigr
\}& \text{if }1<\nu <\frac {\alpha}{2},
\\
\exp \biggl\{- \frac{\Cr{ctwopoint} \llvert a \rrvert ^{\frac{2}{\nu}}d(x,y)}{\log ( \llvert a \rrvert ^{-1})\vee 1} \biggr\}& \text{if }\nu =
\frac {\alpha}{2}. \end{cases} %
\end{equation}
\end{The}
In view of Theorem~\ref{thm:2pointUB}, the bounds
\eqref{eq:lowerboundtau} and \eqref{eq:bound2point} now exactly match up
to constants, when $\nu =1$, and match up to logarithmic correction in
$|a|^{\frac{2}{\nu}}d(x,y)$ when $1<\nu <\frac{\alpha}{2}$. We believe
that our upper bound \eqref{eq:bound2point} is sharp, that is, one can
remove the term $(\log (|a|^{\frac{2}{\nu}}d(x,y)))^{\nu -1}$ in
\eqref{eq:lowerboundtau}; we hope to return to this elsewhere. More importantly,
as opposed to \eqref{eq:bound2point1}, the upper bounds in
\eqref{eq:bound2point} are functions of $|a|^{\frac{2}{\nu}}d(x,y)$ for
all $ 1 \leq \nu < \frac{\alpha}{2}$. Thus, combining
\eqref{eq:lowerboundtau}, \eqref{eq:bound2point}, and
\cite{DrePreRod5}, Theorem~1.7, for the lower bound, when
$\nu =\frac{\alpha}{2}$, and defining
%
\begin{equation}
\label{eq:defxi} \xi \equiv \xi (a)\stackrel{\text{def.}} {=} \llvert a
\rrvert ^{-\frac{2}{\nu}} \quad \text{(with $\xi (0)\stackrel{\text{def.}}{=} \infty $)}, 
\end{equation}
this entails that $\xi $ is the length scale after which the ratio
$\tau _{a}^{\mathrm{tr}}/\tau _{0}^{\mathrm{tr}}$ starts to decay to zero rapidly
(with corrections of order $\log (\xi )$ when
$\nu =\frac{\alpha}{2}$). In particular, it follows that the critical exponent
$\nu _{c}$ associated to the correlation length is equal to
$\frac{2}{\nu}$ for all $1\leq \nu \leq \frac{\alpha}{2}$, which extends
the results from \cite{DrePreRod5}, (1.28). Let us also emphasize that,
contrary to \eqref{eq:bound2point1} or the bounds in case $\nu <1$ mentioned
below \eqref{eq:bound2point1} (cf.~also the disconnection results of
\cite{MR3417515}), the upper bounds in \eqref{eq:bound2point} have a functional
dependence on the parameter $a$, which is \textit{not} Gaussian (not even
when $\nu =1$).

When $G=\mathbb{Z}^{3}$, following ideas from
\cite{GRS21} (see also \cite{MuiSev} for related results), one can actually
strengthen this even more and (almost) match the constants
$\Cr{Ctaulower}$ from \eqref{eq:lowerboundtau} and $\Cr{ctwopoint}$ from
\eqref{eq:bound2point}. More precisely, denoting by ${S}^{2}$ the two-dimensional
unit sphere and abbreviating for each $x\in{\mathbb{R}^{3}}$ by
$[x]$ the vertex closest to $x$ in $\mathbb{Z}^{3}$ (with some arbitrary
choice if $x$ is equidistant to two or more vertices of
$\mathbb{Z}^{3}$), the following result is proved at the end of Section~\ref{sec:2point}.

\begin{Cor}
\label{cor:2pointZ3}
Let $G=\mathbb{Z}^{3}$. There exists $c\in (0,\infty )$, and for all
$\eta \in{(0,1)}$, there exists $C=C(\eta ) \in (0,\infty )$ such that,
for all $e\in{S^{2}}$, $a\in{\mathbb{R}}$ with $|a|\leq c$, and all
$\lambda \geq C$,
%
\begin{equation}
\label{eq:2pointZ3} -\frac{\pi}{6}(1+\eta )\frac{\lambda}{\log (\lambda )}-C\log \log
\xi \leq \log \biggl( \frac{\tau ^{\mathrm{tr}}_{a}(0,[\lambda \xi e])}{\tau ^{\mathrm{tr}}_{0}(0,[\lambda \xi e])} \biggr)\leq -\frac{\pi}{6}(1-
\eta )\frac{\lambda}{\log (\lambda )}. 
\end{equation}
\end{Cor}

The expected rotational invariance at criticality is manifested in
\eqref{eq:2pointZ3} by the fact that the bounds obtained are functions
of $\lambda $ alone (with $\log \log $ corrections) and do not depend on
the choice of $e\in{S^{2}}$. For the related (but harder) model where one
considers excursion sets of the discrete free field~\cite{MR914444,MR3053773},
bounds witnessing a degree of rotational invariance similar to
\eqref{eq:2pointZ3} were derived in \cite{GRS21}, but only asymptotically
in the limit $\lambda \rightarrow \infty $ for fixed parameter $a$. We
refer to \cite{Smi01,duminilcopin2020rotational} and references therein
for rotational invariance results at criticality when $\alpha =2$.

Summing the bounds \eqref{eq:lowerboundtau} and
\eqref{eq:bound2point} over all $y\in{G}$, one also obtains bounds on the
average volume of a (bounded) cluster at level $a \neq 0$, and one can
furthermore deduce from Theorem~\ref{T:1arm-crit} bounds on the tail of
the volume of the critical cluster. This generalizes results from
\cite{DrePreRod5}, Corollary~1.5 and~1.6, for $\nu \leq 1$ to the regime
$1< \nu <\frac{\alpha}{2}$, with improved logarithmic corrections when
$\nu =1$. Let $|K|$ denote the cardinality of a set $K\subset G$.

\begin{Cor}%
\label{cor:vol}
There exist $c, C \in (0,\infty )$ such that, for all $x\in{G}$ and
$n\geq 1$, denoting by $\mathcal{K}^{a}$ the open cluster of $x$ at level
$a$, under the assumption \eqref{eq:restr-params} on $\nu $ and
$\alpha $ and with $q(\cdot )$ as in \eqref{eq:const-UB-1arm-crit},
%
\begin{equation}
\label{eq:criticalvolume} \mathbb{P} \bigl( \bigl\llvert \mathcal{K}^{0}
\bigr\rrvert \geq n \bigr)\leq Cn^{-\frac{\nu}{2\alpha -\nu}}q(n). 
\end{equation}
Moreover, for all $a\neq 0$, if $\nu <\frac{\alpha}{2}$ and
$d=d_{\mathrm{gr}}$, one has
%
\begin{equation}
\label{eq:nearcriticalvolume} c \llvert a \rrvert ^{-\frac{2\alpha}{\nu}+2}\exp \bigl\{-Cq(\xi )
\bigr\}\leq \mathbb{E} \bigl[ \bigl\llvert \mathcal{K}^{a} \bigr
\rrvert 1 \bigl\{ \bigl\llvert \mathcal{K}^{a} \bigr\rrvert <\infty
\bigr\} \bigr]\leq C \llvert a \rrvert ^{-
\frac{2\alpha}{\nu}+2}, 
\end{equation}
and if $\nu <3$ is additionally fulfilled, then
%
\begin{equation}
\label{eq:gamma} \gamma \stackrel{\mathrm{def.}} {=}-\lim_{a\rightarrow 0}
\frac{\log \mathbb{E}[ \llvert \mathcal{K}^{a} \rrvert 1\{ \llvert \mathcal{K}^{a} \rrvert <\infty \}]}{\log  \llvert a \rrvert }= \frac{2\alpha}{\nu}-2. 
\end{equation}
\end{Cor}
For instance, \eqref{eq:gamma} applies when $\mathcal{G}$ is the four-dimensional
Sierpinski carpet; see \cite{DrePreRod2}, Remark~3.10,1), as to why. The
bound \eqref{eq:criticalvolume} can be deduced from Theorem~\ref{T:1arm-crit}
similarly as \cite{DrePreRod5}, Corollary~1.6, and it in particular implies
that the critical exponent $\delta $ (defined in such a way that
$n^{-1/\delta}$ controls the tail in \eqref{eq:criticalvolume}) is smaller
than or equal to $\frac{2\alpha}{\nu}-1$ when it exists. The bound
\eqref{eq:nearcriticalvolume} can be deduced from
\cite{DrePreRod5}, (8.3), and \eqref{eq:bound2point} similarly as
\cite{DrePreRod5}, Corollary~1.5, and \eqref{eq:gamma} follows readily
noting that $q(\xi )/\log (\xi )\rightarrow 0$ as
$\xi \rightarrow \infty $ when $\nu <3\wedge \frac {\alpha}{2}$; see
\eqref{eq:const-UB-1arm-crit}. If $\nu =\frac{\alpha}{2}$, one can still
bound the average near-critical volume in
\eqref{eq:nearcriticalvolume} from above by
$C|a|^{-\frac{2\alpha}{\nu}+2}(\log (|a|^{-1}))^{\alpha -\nu}$, and hence
$\gamma \geq \frac{2\alpha}{\nu}-2$ when it exists, but we cannot prove
a matching lower bound anymore since there is no bound similar to
\eqref{eq:lowerboundtau} currently available in this regime.

 We now comment on the proofs. A key role in deriving the upper bounds in
Theorems~\ref{T:1arm-crit} and~\ref{thm:2pointUB} is played by a certain
obstacle set $\mathcal{O}$ with ``good'' properties, introduced in its
general form in the paragraph leading to Lemma~\ref{lem:hittingobstacle}
below. The relevant obstacle set $\mathcal{O}$, which is a carefully chosen
part of $E^{\geq a}$, is different in each proof and described below. Incidentally,
obstacle sets with similar visibility requirements on the random walk (cf.~Lemma~\ref{lem:hittingobstacle}
below and \cite{sharpnessRI1}, (1.9)) were key in deriving the ``near-critical''
couplings for finite range interlacements, which were instrumental in the
recent proof of sharpness for the vacant set of random interlacements
\cite{sharpnessRI2,sharpnessRI1,sharpnessRI3}, another model with a similar
correlation structure as $(E^{\geq a})_{a \in \mathbb{R}}$. The interlacement
set will also play a role in this article, as explained below.

For the purpose of proving Theorem~\ref{T:1arm-crit}, the obstacle set
$\mathcal{O}$ will be made up of a selection of ``big'' loops stemming
from a loop soup at critical intensity; see \eqref{eq:pf-main1-2}. These
big loops are hard to avoid for the cluster of the origin and, when hit,
ensure that the latter is sufficiently ``fat'' when measured in terms of
its capacity. The pertinence of the loops stems from the metric version
of Le Jan's isomorphism \cite{zbMATH05757657,MR3502602}. A key input (see
Proposition~\ref{P:N-bad}) is a sufficiently strong quantitative control
guaranteeing the presence of big loops. This estimate is the main driving
force of the proof.

For the purposes of proving Theorem~\ref{thm:2pointUB}, the relevant obstacle
set $\mathcal{O}$ comprises pieces of interlacement trajectories in
$\mathcal{I}^{u}$, with $u=\frac{a^{2}}{2}$ and $a$ as appearing in Theorem~\ref{thm:2pointUB}.
Interlacements enter by means of Sznitman's isomorphism theorem
\cite{MR3492939}. They appear when $a \neq 0$ and witness the emergence
of an infinite cluster. Within the proof, interlacements allow, together
with a version of Lupu's formula involving killing on $\mathcal{O}$, to
control the two-point function when $\mathcal{O}$ is ``good;'' see, in
particular, \eqref{eq:bound2pointviacoupling} and Corollary~\ref{cor:2pointkilled}.
A key input to control the goodness features of $\mathcal{O}$ that eventually
give rise to the precise bounds in Theorem~\ref{thm:2pointUB} comes via state-of-the-art
coarse-graining techniques, quantitative in both the parameter $a$ and
the relevant spatial scales, which were developed and progressively refined
in \cite{MR3417515,GRS21,AndPre,prevost2023passage}, the last of which
will be applied here.

We now briefly comment on the relevance of the condition~\eqref{eq:restr-params}.
In the proofs of both Theorems~\ref{T:1arm-crit} and~\ref{thm:2pointUB},
the obstacle set $\mathcal{O}$ is composed of random walk-type objects,
either comprising big loops or pieces of interlacement trajectories. Consequently,
whenever a cluster of large radius intersects $\mathcal{O}$ at a vertex
$x\in{G}$, its capacity in the vicinity of $x$ is not only larger than
the capacity of a line (as needed to satisfy the large radius constraint),
but also than the capacity of a random walk. Whenever $\nu \geq 1$, the
capacity of the latter is typically much larger than the capacity of a
line: see, for example,~\cite{DrePreRod2}, Lemmas~4.4 and~A.1. This boosts
the capacity of this cluster upon intersection with $\mathcal{O}$. Furthermore,
when $\nu \leq \frac {\alpha}{2}$, the capacity of a random walk in a ball
is typically close to the capacity of this ball, and hence any cluster
with large radius locally resembles a ball if measured in terms of capacity
whenever it intersects $\mathcal{O}$. This key random walk input explains
our standing assumption \eqref{eq:restr-params}.

We now describe the organization of this article. Section~\ref{sec:pre}
introduces the framework, some notation, and the important notion of (good)
obstacle set $\mathcal{O}$. Section~\ref{sec:loops} gathers preliminary
estimates on the loop measure for certain sets of interest, at the level
of generality needed for the present article. In combination with the notion
of obstacle set developed in Section~\ref{sec:pre}, these estimates are
used in the proof of Theorem~\ref{T:1arm-crit}, which is presented in Section~\ref{sec:1-arm}.
The proofs of Theorem~\ref{thm:2pointUB}, Corollary~\ref{cor:2pointZ3},
and \eqref{eq:lowerboundtau} are given in Section~\ref{sec:2point}.

Throughout this article we always tacitly assume that
\eqref{eq:restr-params} holds. All constants belong to $(0,\infty )$ and
may implicitly depend on $\nu $ and $\alpha $ in the sequel (as well as
on the constants from \eqref{eq:intro_sizeball},
\eqref{eq:intro_Green}, and the ellipticity condition on controlled weights)
but not on the choice of base point $0$ in \eqref{eq:0} below, which is
arbitrary. Any further dependence on parameters will be stated explicitly.
Numbered constants $c_{1},C_{1},\dots $ are fixed once they first appear,
whereas the constants $c$ and $C$ may change from place to place. It will
be convenient to abbreviate
%
\begin{equation}
\label{eq:bone} b_{1}= %
\begin{cases} 1 &
\text{if }\nu =1,
\\
0 & \text{else} \end{cases} %
\quad \text{and} \quad
b_{2}= %
\begin{cases} 1 & \text{if }
\alpha =2\nu ,
\\
0 & \text{else.} \end{cases} %
\end{equation}

\textit{Note added in proof}. During the review process of this article,
the following improvements have been obtained for the critical one-arm
probabilities as well as the volume:
\begin{itemize}
\item[{(a)}] The techniques developed in this article are crucial to our follow-up
work \cite{DrePreRod10}, Theorem~1.1, where it is shown that, if
$0< \nu < \frac{\alpha}{2}$, one has
%
\begin{equation}
\label{eq:upToConst} cR^{-\frac{\nu}{2}} \leq \psi (R) \leq CR^{-\frac{\nu}{2}}\quad
\text{for all $R \geq 1$.} 
\end{equation}
In particular, this covers the physical benchmark case of
$\mathbb{Z}^{3}$.
As a consequence, one can improve Corollary~\ref{cor:2pointZ3} by removing
the $\log \log \xi $ term in \eqref{eq:2pointZ3}.
\item[{(b)}] Cai and Ding
\cite{cai2024onearmprobabilitiesmetricgraph}, Theorem~1.1, extended
\eqref{eq:upToConst} to $\mathbb{Z}^{4}$ and $\mathbb{Z}^{5}$ by use of
different methods; furthermore, they showed that in $\mathbb{Z}^{6}$ one
has $\psi (R) = R^{-2+o(1)}$ as $R \to \infty $.
\item[{(c)}] Most recently, in \cite{DrePreRod11}, Theorem~1.1, display
\eqref{eq:criticalvolume} has been improved to
\begin{equation*}
cn^{-\frac{\nu}{2\alpha - \nu}} \leq \mathbb{P} \bigl( \bigl\llvert \mathcal{K}^{0}
\bigr\rrvert \geq n \bigr) \leq Cn^{-\frac{\nu}{2\alpha - \nu}},
\end{equation*}
for all $0 < \nu < \frac{\alpha}{2}$ and on $\mathbb{Z}^{4}$ and
$\mathbb{Z}^{5}$; see also \cite{cai2024quasimulti}, Theorem~1.11, in the
lattice case. One can also set $q(\cdot )\equiv 1$ in the lower bound in
\eqref{eq:nearcriticalvolume}, when $0<\nu <\alpha /2$, and extend this
to $\mathbb{Z}^{4}$ and $\mathbb{Z}^{5}$ (with $\alpha =d=\nu +2$); see
\cite{DrePreRod11}, Corollaries 1.4 and 2.3. Note however that the lower
bound \eqref{eq:nearcriticalvolume} is obtained in \cite{DrePreRod11} by
integrating a suitable lower bound on the tail of the near-critical volume,
obtained therein, rather than on the two-point function, as in the present
article.
\end{itemize}

\section{Preliminaries}
\label{sec:pre}

In this section we collect a small amount of notation and gather a useful
preliminary result on the connectivity function for the free field in the
presence of a certain obstacle (set) $\mathcal{O} \subset G$ on which the
random walk is killed; see Corollary~\ref{cor:2pointkilled}. The obstacle
set has certain good properties, which make it hard for the random walk
to avoid; see Lemma~\ref{lem:hittingobstacle}. The related results will
play an important role for the proofs of both Theorems~\ref{T:1arm-crit}
and~\ref{thm:2pointUB}.

We start by discussing a few consequences of our setup from Section~\ref{sec:intro};
see above \eqref{eq:intro_sizeball}. In the sequel, in order to simplify
notations, we typically consider
%
\begin{equation}
\label{eq:0} 0, \quad \text{an arbitrary point in $G$} 
\end{equation}
and abbreviate $B_{r}=B(0,r)$. All results hold uniformly in the choice
of $0$. Similarly as in \cite{DrePreRod2}, Lemma~6.1, we introduce under
our assumptions on $\mathcal{G}$ (note, in particular, that controlled
weights (see above \eqref{eq:intro_sizeball}) correspond to the condition
denoted by $(p_{0})$ in \cite{DrePreRod2}, Section~2) the approximate renormalized
lattice $\Lambda (L)$ with the properties that there is a constant
$\Cl{CLambda}<\infty $ such that, for all $x\in{G}$, $L\geq 1$ and
$N\geq 1$,
%
\begin{equation}
\label{eq:defLambda} %
\begin{aligned} &0 \in \Lambda (L),\quad \bigcup
_{y\in{\Lambda (L)}}B(y,L)=G, \qquad \text{the balls }B \biggl(y,
\frac {L}{2} \biggr), \quad y\in{\Lambda (L)}, \quad\! \text{are disjoint,}\hspace*{-7pt}
\\
&\text{and } \bigl\llvert \Lambda (L)\cap B(x,LN) \bigr\rrvert \leq
\Cr{CLambda}N^{\alpha}. \end{aligned} %
\end{equation}
The asserted disjointness follow from an inspection of the proof of
\cite{DrePreRod2}, Lemma~6.1. We call $\pi =(x_{i})_{1\leq i\leq M}$ a
path in $\Lambda (L)$ from $0$ to $A\subset G$ if $x_{1}=0$,
$x_{M}\in{A}$, $x_{i}\in{\Lambda (L)}$ for all $1\leq i\leq M$, and for
each $1\leq i\leq M-1$, there exist $x\in{B(x_{i},L)}$ and
$y\in{B(x_{i+1},L)}$ such that $x$ and $y$ are neighbors in
$\mathcal{G}$.

We write $X=(X_{t})_{t \geq 0}$ for the diffusion on the cable system (or
metric graph) $\widetilde{\mathcal{G}}$ associated to $\mathcal{G}$ and
$Z=(Z_{n})_{n \geq 0}$ for the discrete skeleton of the trace of $X$ on
$G$; see \cite{DrePreRod3}, Section~2.1, for details. Furthermore, we denote
by $P_{x}$, $x\in \widetilde{\mathcal{G}}$, the canonical law of $X$ with
$X_{0}=x$. If $x\in G$, the law of $Z$ under $P_{x}$ is that of the discrete
time simple random walk on the weighted graph
$\mathcal{G}=(G,\lambda )$. For all
$U\subset \widetilde{\mathcal{G}}$ open and
$x,y\in{\widetilde{\mathcal{G}}}$, we denote by $g_{U}(x,y)$, the Green's
function of $X$ killed outside $U$ between $x$ and $y$, and abbreviate
$g(x,y)=g_{\widetilde{\mathcal{G}}}(x,y)$. When $x,y\in{G}$, then
$g(x,y)$ simply corresponds to the Green's function associated to the walk
$Z$. For $U \subset G$, we write
$H_{U}=H_{U}(Z)=\inf \{ n \geq 0 : Z_{n} \in U\}$ for the entrance time
in $U$ and $T_{U}= H_{G \setminus U}$ for the exit time of $Z$ from
$U$. We also define the equilibrium measure $e_{U}$ of a finite set
$U\subset G$ and its capacity $\operatorname{cap}(U)$ via
%
\begin{equation}
\label{eq:defcap} %
\begin{aligned} \operatorname{cap}(U)=
\operatorname{cap}_{\widetilde{\mathcal{G}}}(U) \stackrel{\text{def.}} {=}\sum
_{x\in{U}}e_{U}(x)\quad \text{where }e_{U}(x)
\stackrel{\text{def.}} {=}\lambda _{x}P_{x}(
\widetilde{H}_{U}=\infty ), \end{aligned} %
\end{equation}
with $\widetilde{H}_{U}$ denoting the first return time to $U$. One then
has
%
\begin{equation}
\label{eq:lastexit} P_{x}(H_{U}<\infty )=\sum
_{y\in{G}}g(x,y)e_{U}(y)\quad \text{for all }x\in{G},
\end{equation}
which follows immediately from a last exit decomposition for the walk;
compare \cite{MR2932978}, (1.57), for the finite graph setting. One further
knows (see, e.g.,~\cite{DrePreRod5}, (5.7)) that, uniformly in
$x \in G$ and $R \geq 1$,
%
\begin{equation}
\label{eq:cap-box} c R^{\nu }\le \operatorname{cap} \bigl(B(x,R) \bigr)
\le CR^{\nu}. 
\end{equation}

Under the standing assumptions on $\mathcal{G}$ (see Section~\ref{sec:intro}),
one also has that the following elliptic Harnack inequality holds. On account
of \cite{DrePreRod2}, (3.3) (cf.~also references therein) such that for
all $\zeta \geq \Cl{zeta0}$, $x \in G$, $R \geq 1$, and
$h: G \to [0,\infty )$, which are harmonic in $B(x,\zeta R)$,
%
\begin{equation}
\label{eq:Harnack} \sup_{y \in B(x, R)} h(y) \leq \Cl[c]{c:Harnack}
\inf_{y \in B(x, R)} h(y). 
\end{equation}

We now introduce the obstacles that will play a role in the sequel. We
call a set $\mathcal{O}\subset G$ a $(L,R,n,\kappa )$-good (or simply good
when the parameters are clear from the context) obstacle if for each path
$\pi $ in $\Lambda (L)$ from $0$ to $B_{R}^{\mathsf{c}}$, there exists
a set $A\subset \operatorname{range}(\pi \cap B_{R})$ with $|A|\geq n$ such that
$\operatorname{cap}(\mathcal{O}\cap B(y,L))\geq \kappa $ for each
$y\in{A}$. Recall our convention regarding constants $c$, $C$ from the end
of the previous section.
%
\begin{lem}
\label{lem:hittingobstacle}
There exists a constant $\Cl{C:dist0} \in (0,\infty )$ such that, for all
$\kappa ,L\geq 1$, $R\geq L$, integer $n\geq 1$,
$x\in{B_{R+\Cr{C:dist0}L}^{\mathsf{c}}}$, and for any
$(L,R,n,\kappa )$-good obstacle set $\mathcal{O}$, one has
\begin{equation*}
P_{0} (H_{x}<H_{\mathcal{O}} )\leq
Cd(x,B_{R+\Cr{C:dist0}L})^{-
\nu}\exp \biggl\{-\frac{c\kappa n}{L^{\nu}}
\biggr\}.
\end{equation*}
\end{lem}
\begin{proof}
By a slight generalization of \eqref{eq:lastexit} (see, for instance,
\cite{prevost2023passage}, (2.17)), there exist constants
$\Cl{Chittingbox}<\infty $ and $\Cl[c]{chittingbox}>0$ such that, for any
$y\in{G}$ with $\operatorname{cap}(\mathcal{O}\cap B(y,L))\geq \kappa $, we have
%
\begin{equation}
\label{eq:hittingproba} P_{z} (H_{\mathcal{O}}<T_{B(y,\Cr{Chittingbox}L)} )
\geq \frac{\Cr{chittingbox}\kappa}{L^{\nu}}\quad \text{for all }z\in{B(y,L)}. 
\end{equation}
Let us introduce recursively a random sequence of vertices
$y_{1},\dots ,y_{M}$ depending on $Z$ under $P_{0}$ as follows:
$y_{i}$ is the first vertex $y\in{\Lambda (L)\cap B_{R}}$ visited by
$B(Z_{k},L)$, $k\in{\mathbb{N}}$ (with an arbitrary rule for splitting
ties) and such that:
\begin{itemize}
\item $\operatorname{cap}(B(y,L)\cap \mathcal{O})\geq \kappa $, and
\item $B(y,L)\cap B(y_{j},\Cr{Chittingbox}L)=\emptyset $ for all
$j\leq i-1$.
\end{itemize}
We stop the recursion the first time $M=M(L,R,\mathcal{O},\kappa )$ after
which such a vertex $y$, as described before, does not exist. Since
$\mathcal{O}$ is a $(L,R,n,\kappa )$-good obstacle, there are at least
$n-1$ different vertices $y\in{\Lambda (L)\cap B_{R}}$ such that $Z$ hits
$B(y,L)$ and $\operatorname{cap}(\mathcal{O}\cap B(y,L)) \ge \kappa $. One can
easily deduce from \eqref{eq:defLambda} that there exists $c>0$ such that
$M\geq cn$ almost surely. Let us denote by $H_{i}$ the first time
$Z$ hits $B(y_{i},L)$ and by $T_{i}$ the first time $Z$ exits
$B(y_{i},\Cr{Chittingbox}L)$ after $H_{i}$. Choosing $\Cr{C:dist0}$ large
enough, we then have
%
\begin{equation}
\label{eq:boundhittingIu} %
\begin{aligned} P_{0}(H_{x}<H_{\mathcal{O}})&
\leq P_{0} \bigl(\mathcal{O} \cap \{Z_{k},H_{i}
\leq k< T_{i}\}=\emptyset \forall 1 \leq i\leq \lceil cn\rceil
,T_{\lceil cn\rceil }\leq H_{x}<\infty \bigr)
\\
&\leq \biggl(1-\frac{\Cr{chittingbox}\kappa}{L^{\nu}} \biggr)^{cn}\cdot
Cd(x,B_{R+
\Cr{C:dist0}L})^{-\nu}; \end{aligned} %
\end{equation}
here the last inequality is obtained as follows: we first apply the strong
Markov property of $Z$ at time $T_{\lceil cn\rceil}$ combined with
\eqref{eq:intro_Green}, and noting that we have
$d(x,Z_{T_{\lceil cn\rceil}})\geq d(x,B_{R+CL})$, in order to obtain the
second factor in the last line of the right-hand side of
\eqref{eq:boundhittingIu}. We then apply the strong Markov property recursively
at times $H_{\lceil cn \rceil -i}$,
$0\leq i\leq \lceil cn \rceil -1$, together with the bound
\eqref{eq:hittingproba}. Using the inequality $1-t\leq e^{-t}$ for all
$t\geq 0$, we obtain the first factor and can conclude.
\end{proof}

We now discuss a consequence of the above result for the Gaussian free
field, which is tailored to our later purposes. Extending the definition
from Section~\ref{sec:intro}, we write ${\varphi}$ from here on to denote
the Gaussian free field on the metric graph
$\widetilde{\mathcal{G}}$ and continue to write $\mathbb{P}$ for its canonical
law. More generally, for $U\subset \widetilde{\mathcal{G}}$ open, we denote
by $\mathbb{P}_{U}$ the law under which $(\varphi _{x})_{x\in{U}}$ is a
Gaussian free field on $U$, that is, a centered Gaussian field with covariance
$g_{U}(x,y)$, $x,y\in{\widetilde{\mathcal{G}}}$; we refer to
\cite{MR3502602} and \cite{DrePreRod3} for further details regarding the
cable system Gaussian free field. We further write ``$x
\leftrightarrow y$ in $A$'' if there is a continuous path in
$A\subset \widetilde{\mathcal{G}}$ from $x$ to $y$, which is consistent
with the notation introduced below \eqref{eq:2-point} when identifying
edges with their respective cables in $\widetilde{\mathcal{G}}$.

We now present a consequence of Lemma~\ref{lem:hittingobstacle} and
\cite{MR3502602}, Proposition~5.2, which is one of the key observations
for the proofs of Theorems~\ref{T:1arm-crit} and \ref{thm:2pointUB}.
%
\begin{Cor}
\label{cor:2pointkilled}
For all $\kappa ,L\geq 1$, $R\geq L$, $n\in{\mathbb{N}}$,
$x\in{G\setminus B_{R+\Cr{C:dist0}L}}$,
$U\subset \widetilde{\mathcal{G}}$ open such that
$\mathcal{O}\stackrel{\mathrm{def.}}{=}G\cap U^{\mathsf{c}}$ is a
$(L,R,n,\kappa )$-good obstacle, we have
\begin{equation*}
\mathbb{P}_{U} \bigl(0\leftrightarrow x\text{ in }\{y\in{U}:\varphi _{y}
\geq 0\} \bigr)\leq Cd(x,B_{R+\Cr{C:dist0}L})^{-\nu}\exp  \biggl\{-
\frac{c\kappa n}{L^{\nu}} \biggr\}.
\end{equation*}
\end{Cor}
\begin{proof}
If $0\in{\mathcal{O}}$ or $x\in{\mathcal{O}}$, the statement is trivial.
Otherwise, by the symmetry of the Gaussian free field and
\cite{MR3502602}, Proposition~5.2, for the graph $\mathcal{G}$ with infinite
killing on $\mathcal{O}$, we have
%
\begin{equation}
\label{eq:PUarcsin}
\mathbb{P}_{U} \bigl(0\leftrightarrow x\text{ in }\{y\in{U}: \varphi _{y}
\geq 0\} \bigr)=\frac{1}{\pi}\arcsin  \biggl(
\frac{g_{U}(0,x)}{\sqrt{g_{U}(0,0)g_{U}(x,x)}} \biggr).
\end{equation}
Using the inequality $\arcsin (t)\leq \pi t/2$ for all $t\in{[0,1]}$, the
right-hand side is upper bounded by
\begin{equation*}
\frac{g_{U}(0,x)}{2 \sqrt{g_{U}(0,0)g_{U}(x,x)}}.
\end{equation*}
Now, note that
$g_{U}(0,x)=g_{U}(x,x)P_{0}(H_{x}<H_{U^{\mathsf{c}}})$ and
$g_{U}(0,x)=g_{U}(0,0)P_{x}(H_{0}<H_{U^{\mathsf{c}}})$. Combining this
with the fact that
$P_{0}(H_{x}<H_{U^{\mathsf{c}}})\leq P_{0}(H_{x}<H_{\mathcal{O}})$ and
that
\begin{equation*}
P_{x}(H_{0}<H_{U^{\mathsf{c}}})\leq
P_{x}(H_{0}<H_{\mathcal{O}})= \lambda
_{0}P_{0}(H_{x}<H_{\mathcal{O}})/
\lambda _{x}\leq CP_{0}(H_{x}<H_{
\mathcal{O}}),
\end{equation*}
which is due to \cite{MR2932978}, (1.23), and
\cite{DrePreRod2}, (2.10), one deduces that the right-hand side of
\eqref{eq:PUarcsin} is bounded by $CP_{0}(H_{x}<H_{\mathcal{O}})$. The
conclusion now follows using Lemma~\ref{lem:hittingobstacle}.
\end{proof}

\section{Markovian loops}
\label{sec:loops}

We now collect some useful properties concerning Markovian loop soups,
as introduced, for instance, in
\cite{MR2045953,zbMATH05120529,zbMATH05757657,MR3502602}. These properties
will then be used in the next section to prove Theorem~\ref{T:1arm-crit}.
Indeed, the upper bound \eqref{eq:1-arm-UB-crit} relies on a profound link
between the Gaussian free field and a Poisson cloud of Markovian loops,
which we now recall. The (rooted) loop soup ${\mathcal{L}}$ of parameter
$\frac{1}{2}$ is a Poisson point process of (bounded, continuous and rooted)
loops on $\widetilde{\mathcal{G}}$ under $\mathbb{Q}$ having intensity
measure $\alpha \mu $, with $\alpha =\frac{1}{2}$ and a measure
$\mu $, which we proceed to introduce. The measure $\mu $ acts on the space
of rooted loops, that is,~of continuous trajectories
$\gamma : [0,T] \to \widetilde{\mathcal{G}}$ satisfying
$\gamma (0)= \gamma (T)$, for some $T=T(\gamma ) \in (0,\infty )$ called
the duration of the loop. More precisely, we define
%
\begin{equation}
\label{eq:loops-mu} \mu ( \cdot ) \stackrel{\text{def.}} {=}
\int _{
\widetilde{\mathcal{G}}} {\mathrm{d}}m(x)
\int _{0}^{\infty} \frac{{\mathrm{d}}t}{t}
q_{t}(x,x) P_{x,x}^{t}( \cdot ) ,
\end{equation}
where $P_{x,x}^{t}( \cdot  )$ is the time $t$ bridge probability measure
for the diffusion $X$ on $\widetilde{\mathcal{G}}$ (introduced above
\eqref{eq:lastexit}), $m$ is the natural Lebesgue measure on
$\widetilde{\mathcal{G}}$, which assigns length
$1/(2\lambda _{x,y})$ to the cable joining $x$ and $y$, and $q_{t}$ is
the transition density of $X$ relative to $m$; see
\cite{MR3238780}, Section~2, and \cite{MR3502602}, Section~2, where it is
denoted by $\widetilde{\mathcal{L}}_{\nicefrac12}$, for details.

If $U \subset \widetilde{\mathcal{G}}$ is an open subset, the loop soup
$\mathcal{L}_{U}$ is obtained by retaining only the loops in
$\mathcal{L}$ whose range is contained in $U$. The restriction property
of $\mathcal{L}$ (see \cite{MR3238780}, Theorem~6.1, or also
\cite{MR2932978}, Proposition~3.6, in the discrete setting) asserts that
%
\begin{equation}
\label{eq:LS-rest} {\mathcal{L}}_{U} \text{ has law }
\mathbb{Q}_{U}, 
\end{equation}
where $\mathbb{Q}_{U}$ now denotes the law of the metric graph loop soup
with underlying graph $U$, that is, with infinite killing on
$U^{\mathsf{c}}$. The loop soup $\mathcal{L}$ induces clusters in
$\widetilde{\mathcal{G}}$ as follows. Two continuous loops belong to the
same cluster if there exists a finite sequence of loops starting and ending
with the two loops of interest, and such that the ranges of any two consecutive
loops in the sequence intersect. Clusters of loops are connected components
obtained in this way using loops in the support of $\mathcal{L}$. We will
take advantage of the following link relating the loop soup
$\mathcal{L}$ to the free field $\varphi $ under $\mathbb{P}$. Considering
%
\begin{equation}
\label{eq:C_s} \mathcal C= \mathcal{C}({\mathcal{L}}),\quad  \text{the trace on $\widetilde{\mathcal{G}}$ of the cluster of loops (in
$\mathcal{L}$) containing $0$},
\end{equation}
one knows (see, for instance, \cite{MR3502602}, Proposition~2.1, for a proof,
which relies on an isomorphism of Le Jan; see
\cite{zbMATH05757657}, Theorem~13) that
%
\begin{equation}
\label{eq:iso-LS} \text{the cluster of $0$ in
$\bigl\{x\in{\widetilde{\mathcal{G}}}:  \llvert \varphi _{x} \rrvert >0\bigr\}$ has the same law
under $\mathbb{P}$ as $\mathcal C$ under $\mathbb{Q}$.} 
\end{equation}
Further, notice that by \cite{DrePreRod3}, Lemma~4.1, if
$\varphi _{0}>0$. then the closure of the cluster of $0$ in
$\{x\in{\widetilde{\mathcal{G}}}: |\varphi _{x}|>0\}$ is almost surely
equal to the cluster of $0$ in
$\{x\in{\widetilde{\mathcal{G}}}: \varphi _{x}\geq 0\}$, whose intersection
with $G$ has the same law as the cluster of $0$ in $E^{\geq 0}$; see below
\eqref{eq:bond-proba}, as explained around \cite{DrePreRod5}, (1.6). Combined
with \eqref{eq:iso-LS}, this explains the relevance of $\mathcal{L}$ in
our context.

We proceed to gather specific features of the latter that will be beneficial
for us. To this effect, it is convenient to consider the discrete time
loop soup $\widehat{\mathcal{L}}$ on $G$ induced by $\mathcal{L}$, which
is obtained by considering the trace on $G$ of all continuous loops in
the support of $\mathcal{L}$ intersecting $G$, and only retaining nontrivial
loops, that is,~any loop visiting more than one vertex of
$\mathcal{G}$. As explained in \cite{MR3502602}, Section~2, combined with
the same proof as in the finite setting of \cite{MR2932978}, (3.17), the
intensity measure of $\widehat{\mathcal{L}}$ can be described as follows.
To any rooted loop $\gamma $ on $\widetilde{\mathcal{G}}$ visiting at least
two vertices of $G$, one associates $N=N(\gamma ) $ the number of times
($\geq 2$) $\gamma $ jumps to another vertex in $G$ and the corresponding
discrete skeleton $Z_{0}=Z_{0}(\gamma ),\ldots,Z_{n}=Z_{n}(\gamma )= Z_{0}$
when $N(\gamma )=n$. Then for all $n \geq 2$ and
$x_{0},\dots , x_{n-1} \in {G}$,
%
\begin{equation}
\label{eq:loops-mu-disc} \mu (N=n , Z_{0} = x_{0},\dots
Z_{n-1}= x_{n-1}) = \frac{1}{n} \prod
_{0
\leq i < n} \frac{\lambda _{x_{i},x_{i+1}}}{\lambda _{x_{i}}}  \quad \text{(with $x_{n}=x_{0}$).}
\end{equation}

We now collect several useful properties of the measure $\mu $ in
\eqref{eq:loops-mu}. In the special case of $\mathbb{Z}^{\alpha}$,
$\alpha \geq 3$ integer, results of this kind have
appeared in \cite{zbMATH06566372}, Section~2. These relied in the case
$\alpha =4$ on intersection results of \cite{zbMATH03791381} for two random
walks. With hopefully obvious notation, we write
$ K \stackrel{\gamma}{\longleftrightarrow} U $ to denote the property that
the loop $\gamma $ intersects both $K$ and $U$, for
$K,U \subset {G}$. Furthermore, recall that $b_{2}$ has been defined in
\eqref{eq:bone}.

\begin{lem}[$0< \nu \leq \frac{\alpha}{2} $]
\label{L:loop-intensities}
For all $\zeta \geq C$, $R \geq 1$, $x \in {G}$, and
$K \subset B(x,R)$,
%
\begin{align}
 \label{eq:mu-UB} &\mu \bigl(K \stackrel{\gamma} {\longleftrightarrow} B(x,
\zeta R)^{
\mathsf{c}} \bigr) \leq C \cdot \operatorname{cap}(K) (\zeta
R)^{-\nu}, 
\\
 \label{eq:mu-LB} &\mu \bigl( K \stackrel{\gamma} {\longleftrightarrow} B(x,
\zeta R)^{
\mathsf{c}}, \operatorname{cap}(\gamma ) \geq c \cdot
{R^{\nu}} { ( \log R)^{-b_{2}}} \bigr) \geq c \cdot
\operatorname{cap}(K) (\zeta R)^{-
\nu}, 
\end{align}
where with a slight abuse of notation
$\operatorname{cap}(\gamma ) = \operatorname{cap}(\operatorname{range}(
\gamma )\cap G)$.
\end{lem}

\begin{proof}
We start by showing \eqref{eq:mu-UB}. Abbreviate $B=B(x, \zeta R)$, and
write $Z=Z(\gamma )$ for the discrete skeleton of the loop $\gamma $. Let
$R_{0} = H_{K}= \inf \{ n \geq 0 : Z_{n} \in K\}$, with the convention
$\inf \emptyset = \infty $, and for $k \geq 0$, let
$D_{k+1} = R_{k} + T_{B} \circ \theta _{R_{k}}$, where we recall the notation
$T_{B}=H_{B^{\mathsf{c}}}$, when $R_{k} < \infty $ and
$D_{k}= \infty $ otherwise, and
$R_{k+1}= D_{k+1} + H_{K} \circ \theta _{D_{k+1}} $ (possibly infinite).
The random variable
$\kappa = \kappa (\gamma )= |\{ k \geq 1 : R_{k} \leq N(\gamma ) \}|$,
which denotes the number of times $\gamma $ returns to $K$, is always positive
(and finite) on the event appearing on the left-hand side of
\eqref{eq:mu-UB}. Thus, using \eqref{eq:loops-mu-disc}, decomposing over
the value of $\kappa =k$ and rerooting the (discrete, rooted) loop
$Z$ at $Z_{R_{i}}$, where $i\in{\{1,\dots ,k\}}$ is chosen uniformly at
random, which produces a factor $\frac{n}{k}$, one obtains that
%
\begin{equation}
\label{eq:mu-UB-pf1} \mu \bigl( K \stackrel{\gamma} {\longleftrightarrow} B(x, \zeta
R)^{
\mathsf{c}} \bigr) = \sum_{k \geq 1}
\frac{1}{k} \sum_{y \in \partial K}
P_{y}(R_{k}<\infty , Z_{R_{k}}=y),
\end{equation}
with $\partial K$ denoting the interior boundary of the set $K$. For
$k \geq 1$, one has for all $y \in B(x,R)$, applying the Markov property
at time $D_{k-1}$ and noting that, for $z \in K$ the function
$u \mapsto P_{u}(R_{1}<\infty ,Z_{R_{1}}=z)$ is harmonic in $B$,
%
\begin{align}
\label{eq:mu-UB-pf2} %
\begin{aligned} &P_{y}(R_{k}<
\infty , Z_{R_{k}}=z)
\\
&\quad = \sum_{u \in \partial K} P_{y}(R_{k-1}<
\infty , Z_{R_{k-1}}=u)P_{u}(R_{1}< \infty
,Z_{R_{1}}=z)
\\
&\quad \!\!\!\stackrel{\text{\eqref{eq:Harnack}}} {\leq} \sum_{u \in \partial K}
P_{y}(R_{k-1}< \infty , Z_{R_{k-1}}=u)
\Cr{c:Harnack} P_{x}(R_{1}<\infty ,Z_{R_{1}}=z)
\\
&\quad \leq \Cr{c:Harnack} P_{y}(R_{k-1}<\infty )
P_{x}(R_{1} < \infty , Z_{R_{1}}= z).
\end{aligned} 
\end{align}
By a straightforward induction argument, one finds that
$P_{y}(R_{k}<\infty ,Z_{R_{k}}=z)$ is bounded from above by
$(\Cr{c:Harnack}P_{x}(R_{1}<\infty ))^{k-1}P_{x}(R_{1}<\infty ,Z_{R_{k}}=z)$.
We now choose $z=y$, sum over $y\in{\partial K}$, and plug the resulting
estimate into \eqref{eq:mu-UB-pf1}. Consequently, using that
$P_{x}(R_{1}<\infty ) \leq \sup_{z \in B^{\mathsf c}} P_{z}(H_{K} <
\infty )$, which is, at most,
$c \cdot \text{cap}(K) (\zeta R)^{-\nu}$ by a last-exit decomposition and
\eqref{eq:intro_Green} (in particular, this implies that
$\Cr{c:Harnack}P_{x}(R_{1}<\infty ) \leq 1-c$ for some $c> 0$ upon possibly
taking $\zeta \geq C$ and hence the convergence of the resulting geometric
series in \eqref{eq:mu-UB-pf1}), the bound \eqref{eq:mu-UB} readily follows.

To deduce \eqref{eq:mu-LB}, one first observes, recalling the discussion
that led to \eqref{eq:mu-UB-pf1}, that $\kappa (\gamma ) \geq 1$ on the
event appearing on the left-hand side of \eqref{eq:mu-LB} if
$\gamma $ is rooted in $K$. To obtain a lower bound, one simply retains
loops $\gamma $ with $\kappa (\gamma )=1$, reroots such loops similarly
as above \eqref{eq:mu-UB-pf1}, and requires the required capacity to be
generated ``on the way back,'' that is,~after time $D_{1}$ and before exiting
the ball of radius $R$ around $Z_{D_{1}}$ (which occurs before hitting
$K$ if $\lambda _{0}>2$), to find that, for all $ t \geq 1$,
%
\begin{align}
\label{eq:mu-UB-pf3} %
\begin{aligned} &\mu \biggl( K \stackrel{\gamma} {
\longleftrightarrow} B(x, \zeta R)^{
\mathsf{c}} , \operatorname{cap}(\gamma )
\geq \frac{R^{\nu}}{t (\log R)^{b_{2}}} \biggr)
\\
&\quad \geq \inf_{z \in \partial _{\text{out}}B } P_{z} \biggl(H_{K}<
\infty , \operatorname{cap} (Z_{[0, T_{B(z,R)}]} ) \geq \frac{R^{\nu}}{t (\log R)^{b_{2}}}
\biggr), \end{aligned} 
\end{align}
where
$Z_{[0,s]}= \{ x \in {G}: Z_{t} = x \text{ for some } t \in [0,s]\} $ and
$\partial _{\text{out}}B=\{x\in{B^{\mathsf{c}}}:\exists  y\in{B}
\text{ with }y\sim x\}$. Applying the Markov property at time
$T_{B(z,R)}$ and noting that $X_{T_{B(z,R)}}\in{B(x, (\zeta +C)R)}$ for
all $z\in{\partial _{\text{out}}B}$ by \cite{DrePreRod2}, (2.8), the probability
in the second line of \eqref{eq:mu-UB-pf3} is bounded from below by
\begin{equation*}
\inf_{z \in G } P_{z} \biggl(\operatorname{cap}
(Z_{[0, T_{B(z,R)}]} ) \geq \frac{R^{\nu}}{t (\log R)^{b_{2}}} \biggr) \cdot \inf
_{z'
\in B(x,(\zeta +C)R)} P_{z'}(H_{K}<\infty ).
\end{equation*}
Regarding the second factor, by a similar computation as above involving
a last-exit decomposition and the lower bounds in
\eqref{eq:intro_Green}, it is bounded from below by
$c \cdot \operatorname{cap}(K)(\zeta R)^{-\nu}$. With respect to the first
factor, observe that, applying \cite{DrePreRod5}, Lemma~5.3, with $K$ as
appearing therein chosen as $K= \emptyset $ and $t= C$ sufficiently large,
one infers that the infimum over $z$ is bounded away from zero by
$c>0$, whence \eqref{eq:mu-LB} follows.
\end{proof}

\section{Critical one-arm probability}
\label{sec:1-arm}

In this section we prove Theorem~\ref{T:1arm-crit}. We start with some
preparation and consider the scales
%
\begin{equation}
\label{eq:R-form} R= \Cl{CKL}(\ell +1) L\quad \text{for } \ell , L \geq
2, 
\end{equation}
with $\Cr{CKL}$ to be chosen momentarily (see \eqref{eq:A_i-prop} below).
Recall the approximate lattice $\Lambda (L)$ from the beginning of Section~\ref{sec:pre}.
We define $\mathcal{A}_{k} \subset \Lambda (L)$ for
$1 \leq k \leq \ell $ as
%
\begin{equation}
\label{eq:A_i} \mathcal{A}_{k} = \bigl\{ x \in \Lambda (L) :
B(x,L) \cap \partial B_{
\Cr{CKL} k L } \neq \emptyset \bigr\}, 
\end{equation}
and the associated ``annulus''
%
\begin{equation}
\label{eq:AA_i} \mathbb{A}_{k}= \bigcup
_{x \in \mathcal{A}_{k}} B(x,L) \ (\subset G). 
\end{equation}
Henceforth, the constant $\Cr{CKL}$ in \eqref{eq:R-form} is fixed so that,
for all $\ell , L \geq 2$ and $k, k' \in \{1, \ldots , \ell \}$,
%
\begin{equation}
\label{eq:A_i-prop} \mathbb{A}_{k} \subset B_{R}\quad
\text{and}\quad  d(\mathbb{A}_{k},\mathbb{A}_{k'})
\geq \bigl\llvert k-k' \bigr\rrvert L, 
\end{equation}
which is always possible by \eqref{eq:A_i} and
\cite{DrePreRod2}, (2.8). We now consider the continuous loop soup
${\mathcal{L}}$ of intensity $ \alpha =\frac{1}{2}$ on
$\widetilde{\mathcal{G}}$ with canonical law ${\mathbb{Q}}$; see Section~\ref{sec:loops}
for details, and in the sequel we refer to as ``loop'' any element in the
support of its intensity measure. For given $\delta >0$ and
$L \geq 1$, we will call a loop \textit{small} if
%
\begin{equation}
\label{eq:small-loop} \text{cap}(\gamma )< \delta L^{\nu} (\log
L)^{-b_{2}}, 
\end{equation}
and \textit{big} otherwise. With this terminology the point measures
${\mathcal{L}}^{\mathsf{b}}_{k}$ for $1\leq k \leq \ell $ are defined as follows.
If ${\mathcal{L}}= \sum_{i} \delta _{\gamma _{i}}$ refers to a generic
realization, then
%
\begin{equation}
\label{eq:LS-D} {\mathcal{L}}_{k}^{\mathsf{b}} =\sum
_{i} \delta _{\gamma _{i}} 1 \bigl\{ \gamma
_{i} \text{ is big and } \operatorname{range}(\gamma _{i})
\subset \widetilde{B}(x,L) \text{ for some } x\in \mathcal{A}_{k}
\bigr\}, 
\end{equation}
where with hopefully obvious notation,
$\widetilde{B}(x,L) \subset \widetilde{\mathcal{G}}$ is obtained by adding
all the cables between neighboring vertices in $B(x,L)$. Recalling from
\eqref{eq:C_s} the cluster ${\mathcal C}$ of the origin in the loop soup
$\mathcal{L}$, we consider the events
%
\begin{equation}
\label{eq:B-i} \mathbf{B}_{k} = \begin{Bmatrix}
\text{$\mathcal{C} \cap B_{R}^{\mathsf{c}} \neq \emptyset $ and
$\mathcal{C}$ does not contain the}
\\
\text{trace of any loop in the support of ${\mathcal{L}}_{k}^{\mathsf{b}}$} \end{Bmatrix}. 
\end{equation}
Note that \eqref{eq:B-i} depends implicitly on the choice of parameters
$\delta $, $\ell $, and $L$ (furthermore, $\mathbf{B}_{k}$ is measurable
since the restrictions in \eqref{eq:small-loop} and \eqref{eq:LS-D} induce
measurable constraints). Intuitively, the event $\mathbf{B}_{k}$ refers
to an annulus in which certain big loops are avoided by
$\mathcal{C}$ (which we think of as being bad). Our interest will be in
the quantity
%
\begin{equation}
\label{eq:N-bad} {N}= {N}_{\delta , \ell ,L} \stackrel{\mathrm{def.}} {=} \sum
_{1
\leq k \leq \ell} 1_{\mathbf{B}_{k}}. 
\end{equation}
The following result is key to the proof of Theorem~\ref{T:1arm-crit}.
It asserts that large families of bad annuli (i.e., $k$'s such that
$\mathbf{B}_{k}$ occurs) are typically rare. Importantly, the quantitative
error bound is sufficiently sharp.

\begin{prop}
\label{P:N-bad}
For all integers $\ell , L \geq 2$,
$\delta , \rho \in (0,\frac{1}{2})$ such that
$ L^{-c} \leq \delta \leq \Cl[c]{C:N-bad-1}(\log \ell )^{-
\frac{\nu}{\alpha}}$ and
$\rho \log (\rho ^{-1})\leq \Cl[c]{C:N-bad-6}\delta (\log L)^{-b_{2}}$,
one has
%
\begin{equation}
\label{eq:N-bad-LD} {\mathbb{Q}} \bigl(N\geq (1-\rho ) \ell \bigr) \leq C (\ell
L)^{\alpha } \exp \biggl\{ - \frac{\Cl[c]{C:N-bad-2}\delta \ell}{(\log L)^{b_{2}}} \biggr\}. 
\end{equation}
\end{prop}

We will prove Proposition~\ref{P:N-bad} further below and first show how
to deduce Theorem~\ref{T:1arm-crit} from it. This involves an estimate
on the capacity of the union of many big loops, such as the cluster
$\mathcal{C}$ on the complement of the event in \eqref{eq:N-bad-LD}. The
following lemma gives a lower bound which is not far from additive.

\begin{lem}%
\label{L:cap-LB}
For all integers $\ell , L \geq 2$, the following holds. Let
$I \subset \{1,\dots , \ell \}$, and assume that
$S_{k} \subset \mathbb{A}_{k}$ with $\text{cap}(S_{k}) \geq \eta $ for each
$k \in I$. Then recalling $b_{1}$ from \eqref{eq:bone}, one has
%
\begin{equation}
\label{eq:cap-loops-LB} \operatorname{cap} \biggl( \bigcup
_{k \in I} S_{k} \biggr) \geq \Cl[c]{ccapcalC}
\llvert I \rrvert \Bigl( L^{\nu}{ \bigl(\log \bigl(1+ \llvert I \rrvert
\bigr) \bigr)^{-b_{1}}} \wedge \inf_{k \in I}
\operatorname{cap}(S_{k}) \Bigr). 
\end{equation}
\end{lem}
\begin{proof}
By a classical variational characterization of capacity (see, for instance,
\cite{MR2932978}, Proposition~1.9, for a proof on finite graphs), for any
$S \subset G$ one has that
%
\begin{equation}
\label{eq:capVar} \operatorname{cap} ( S ) = \biggl(\inf_{\mu}
\sum_{x,y \in G} \mu (x) g(x,y) \mu (y)
\biggr)^{-1}, 
\end{equation}
where the infimum ranges over all probability measures $\mu $ supported
on $S$. Consider the measure
%
\begin{equation}
\label{eq:mu-cap} \mu (x)\stackrel{\mathrm{def.}} {=} \frac{1}{ \llvert I \rrvert }\sum
_{k \in I} \bar{e}_{S_{k}}(x), \quad x \in G,
\end{equation}
where $\bar{e}_{S_{k}}= {e}_{S_{k}}/\text{cap}(S_{k})$ denotes the normalized
equilibrium measure, so $\mu $ is indeed a probability measure supported
on $\bigcup_{k \in I} S_{k}$. Evaluating the sum on the right-hand side
of \eqref{eq:capVar} at $\mu $ from \eqref{eq:mu-cap} and noting that,
for all $x \in S_{k}$, $y \in S_{j}$, $1\leq k\neq j \leq \ell $, the bound
$g(x,y) \leq C d(\mathbb{A}_{k}, \mathbb{A}_{j})^{-\nu} \leq CL^{-\nu}|j-k|^{-
\nu}$ holds owing to \eqref{eq:intro_Green} and \eqref{eq:A_i-prop}, it
follows that
\begin{align*}
\operatorname{cap} \biggl( \bigcup_{k \in I}
S_{k} \biggr)^{-1} &\leq \frac{1}{ \llvert I \rrvert ^{2}} \sum
_{k,j \in I} \sum_{x \in S_{k}}
\sum_{y \in S_{j}} \bar{e}_{S_{k}}(x) g(x,y)
\bar{e}_{S_{j}}(y)
\\
&\leq \frac{1}{ \llvert I \rrvert ^{2}} \sum_{k \in I}
\frac{1}{\text{cap}(S_{k})} + \frac {2}{ \llvert I \rrvert } \sum_{n=1}^{ \llvert I \rrvert }
C(nL)^{-\nu}.
\end{align*}
Here the first term in the last display emanates from the on-diagonal sums
for $k=j$ using \eqref{eq:lastexit} and the definition
$\bar{e}_{S_{k}}= {e}_{S_{k}}/\text{cap}(S_{k})$, while to obtain the second
term one also takes advantage of assumption \eqref{eq:intro_Green}. This
now entails \eqref{eq:cap-loops-LB}.
\end{proof}
Using Proposition~\ref{P:N-bad} and Lemma~\ref{L:cap-LB}, we now give the
proof of our first main result.

\begin{proof}[Proof of Theorem~\ref{T:1arm-crit}]
On account of \eqref{eq:iso-LS} and below, and since $\varphi $ and
$-\varphi $ have the same law under $\mathbb{P}$, one has that
%
\begin{equation}
\label{eq:pf-Th1-1}
\mathbb{P}(0 \leftrightarrow B_{R}^{\mathsf{c}}) = \frac{1}{2}
\mathbb{Q}({\mathcal{C}} \cap B_{R}^{\mathsf{c}} \neq \emptyset ).
\end{equation}
As a consequence, it is sufficient to upper bound the quantity on the right-hand
side of \eqref{eq:pf-Th1-1}. It is further sufficient to assume that
$R \geq C$, which will be tacitly supposed from here on. In particular,
no loss of generality is incurred by assuming that $R$ is of the form
\eqref{eq:R-form} and by proving the statement for a suitable choice of
$\ell >1$ (see \eqref{eq:pf-Th1-params} below) and all $L \geq C$.

As we now explain, for all $\ell >1$, $L \geq C$, $\delta \in (0,c)$, and
$\rho \in (0,1)$, with
%
\begin{equation}
\label{eq:def_eta} \eta =\Cr{ccapcalC}\rho \ell L^{\nu} \biggl(
\frac{\delta }{ (\log L)^{b_{2}} }\wedge \frac{1}{(\log (1+\rho \ell ))^{b_{1}} } \biggr) 
\end{equation}
(see Lemma~\ref{L:cap-LB} regarding $\Cr{ccapcalC}$), one has the inclusion
%
\begin{equation}
\label{eq:pf-Th1-2} \bigl\{{\mathcal{C}} \cap B_{R}^{\mathsf{c}}
\neq \emptyset , \text{cap}(\mathcal{C}) < \eta \bigr\} \subset \bigl\{ N
\geq (1-\rho )\ell \bigr\} 
\end{equation}
with the random variable $N$ as in \eqref{eq:N-bad}. To see this, suppose
that $\{N < (1-\rho ) \ell \}$. We now argue that the intersection of this
event with the one on the left-hand side of \eqref{eq:pf-Th1-2} is empty,
from which the desired inclusion in \eqref{eq:pf-Th1-2} follows. For this
purpose, consider the (random) set
$I =\{k: \mathbf{B}_{k} \text{ does not occur}\} \subset \{1,\dots ,
\ell \}$ so that
%
\begin{equation}
\label{eq:I-LB} \llvert I \rrvert =\ell - N > \rho \ell 
\end{equation}
on the event $\{N < (1-\rho ) \ell \}$. By definition of
$\mathbf{B}_{k}$ in \eqref{eq:B-i} and the defining property of big loops
(see \eqref{eq:small-loop}), one has on the event
$\{ {\mathcal{C}} \cap B_{R}^{\mathsf{c}} \neq \emptyset \}$ that
$\mathcal{C} \supset \bigcup_{k\in I} S_{k}$, where
$S_{k} \subset \mathbb{A}_{k}$ comprises at least the range of one loop
of capacity at least $\delta L^{\nu} (\log L)^{-b_{2}}$ for each
$k \in I$, whence
$\text{cap}(S_{k}) \geq \delta L^{\nu} (\log L)^{-b_{2}} $ by monotonicity
of the capacity. Applying Lemma~\ref{L:cap-LB}, this yields that
\begin{equation*}
\text{cap}(\mathcal{C}) \geq \Cr{ccapcalC} (\ell -N)L^{\nu}
\bigl( \bigl( \delta (\log L)^{-b_{2}} \bigr)\wedge \bigl(\log (1+\ell -N)
\bigr)^{-b_{1}} \bigr) \stackrel{\text{\eqref{eq:I-LB}}, \text{\eqref{eq:def_eta}}} {
\geq} \eta ,
\end{equation*}
and \eqref{eq:pf-Th1-2} follows.

Using \eqref{eq:pf-Th1-1}, \eqref{eq:pf-Th1-2}, the tail bound
$\mathbb{Q}(\text{cap}(\mathcal{C}) > t) \leq C t^{-1/2}$ (which holds by
\eqref{eq:iso-LS}) and \cite{DrePreRod5}, Corollary~1.3, it follows from
\eqref{eq:N-bad-LD} that under the assumptions of Proposition~\ref{P:N-bad},
%
\begin{align}
\label{eq:pf-Th1-3}
\begin{aligned}
\mathbb{P}(0 \leftrightarrow B_{R}^{\mathsf{c}}) &\leq \mathbb{Q}(
\text{cap}(\mathcal{C}) \geq \eta ) + \mathbb{Q}( N \geq (1-\rho )
\ell )
\\
&\leq C  \big( \eta ^{-1/2} + (\ell L)^{\alpha }\exp  \{ -
\Cr{C:N-bad-2} (\log L)^{-b_{2}} \delta \ell  \}  \big).
\end{aligned}
\end{align}
We now choose the parameters $\delta $, $\rho $, and $\ell $ so that they
satisfy the assumptions of Proposition~\ref{P:N-bad} and that the second
line of \eqref{eq:pf-Th1-3} is small. For $\gamma \in (0,1)$ to be specified
in a moment and for any $L>e^{4}$, let
%
\begin{equation}
\label{eq:pf-Th1-params} %
\begin{aligned} \rho &=\gamma (\log
L)^{-b_{2}} (\log \log L)^{-(b_{2}+
\frac{\nu}{\alpha})}(\log \log \log
L)^{-1+b_{2}},
\\
\ell &=\gamma ^{-1} (\log L)^{1+b_{2}}(\log \log
L)^{\nu /\alpha},
\\
\delta & = \Cr{C:N-bad-1} (\log \ell )^{-\nu /\alpha}. \end{aligned}
%
\end{equation}
Observe that $\gamma \log (1/\gamma ) \to 0$ as $\gamma \to 0$. Choosing
$\gamma $ small enough, one can thus ensure that
$\rho \log (1/\rho ) \leq \Cr{C:N-bad-6} \delta (\log L)^{-b_{2}}$ for
all $L \geq C$, and hence by our choice of $\delta $ and $\ell $, the assumptions
of Proposition~\ref{P:N-bad} are satisfied upon possibly further increasing
$L$. In particular, \eqref{eq:pf-Th1-3} is in force and, possibly reducing
the value of $\gamma $ even further so that
$\frac{\Cr{C:N-bad-2}\Cr{C:N-bad-1}}{\gamma}>\alpha +\frac{\nu}{2}$, the
second term in brackets in the second line of \eqref{eq:pf-Th1-3} is
$O(L^{-\nu /2-\varepsilon })$ for some $\varepsilon >0$ as
$L \to \infty $. In view of \eqref{eq:def_eta}, \eqref{eq:pf-Th1-3}, and
\eqref{eq:pf-Th1-params}, this means that the term $ \eta ^{-1/2} $ dominates,
and one obtains from \eqref{eq:def_eta}, noting further that
$\delta (\log L)^{-b_{2}}\geq (\log (1+\rho \ell ))^{-b_{1}}$ if and only
if $\nu = 1$, that for all $L \geq C$,
%
\begin{equation}
\label{eq:end-Th1}
\mathbb{P}(0 \leftrightarrow B_{R}^{\mathsf{c}})
 \leq C (\log \log \log L)^{\frac{1-b_{2}}{2}}(\log \log L)^{
\frac{b_{1}+ b_{2}}{2} + \frac{(1+\nu -b_{1})\nu}{2\alpha}} (\log L)^{
\frac{\nu -1 + b_{2}(\nu +1)}{2}} \cdot (\ell L)^{-\frac{\nu}{2}}.
%
\end{equation}
Recalling the definition of $R$ in \eqref{eq:R-form}, the claim now follows
since one can safely replace $L$ by $\ell L$ inside the $\log $ and
$\log \log $ factors by definition of $\ell $ in
\eqref{eq:pf-Th1-params}, and the exponents of $\log \log R$ in
\eqref{eq:const-UB-1arm-crit} arise using that $\nu < \alpha /2$ in the
second case and that $\alpha \geq \nu +2=3$ when $\nu =1$.
\end{proof}

We now prove Proposition~\ref{P:N-bad}.

\begin{proof}[Proof of Proposition~\ref{P:N-bad}]
Recall the event $\mathbf{B}_{k}$ from \eqref{eq:N-bad}, and for arbitrary
$D \subset \{1,\dots , \ell \}$, consider the event
$\mathbf{B}_{D}=\bigcap_{k \in D} \mathbf{B}_{k}$. For each
$D\subset \{1,\dots ,\ell \}$ with $|D| \geq (1-\rho )\ell $, and
$\delta $, $\ell $, $L $, $\rho $ as appearing in the statement, we will argue
that, if
$ L^{-c}\leq \delta \leq ( \Cr{C:N-bad-1}(\log \ell )^{-
\frac{\nu}{\alpha}})$,
%
\begin{equation}
\label{eq:pf-main1-1} {\mathbb{Q}} (\mathbf{B}_{D} ) \leq C (\ell
L)^{\alpha }\exp \biggl\{ - \frac{c(1-\rho )\delta \ell}{ (\log L)^{b_{2}}} \biggr\}. 
\end{equation}
Once this is shown, applying a union bound over subsets of
$\{1,\dots ,\ell \}$ of cardinality $\lceil (1-\rho )\ell \rceil $ and
using the bound
$\tbinom{\ell}{\lceil (1-\rho )\ell \rceil} \leq e^{C\rho \log (1/
\rho )\ell}$ valid for all $\rho \in (0,\frac{1}{2})$, which is a consequence
of Stirling's formula, the bound \eqref{eq:N-bad-LD} directly follows when
$\rho \log (\rho ^{-1})\leq \Cr{C:N-bad-6}\delta (\log L)^{-b_{2}}$ for
$\Cr{C:N-bad-6}$ small enough.

We now prove \eqref{eq:pf-main1-1}. With ${\mathcal{L}}_{k}^{\mathsf{b}}$ as
in \eqref{eq:LS-D}, let
%
\begin{equation}
\label{eq:LS-D-really} {\mathcal{L}}_{D}^{\mathsf{b}} =\sum
_{k \in D} {\mathcal{L}}_{k}^{\mathsf{b}}.
\end{equation}
In words, ${\mathcal{L}}_{D}^{\mathsf{b}}$ collects all the big loops contained
in a box $\widetilde{B}(x,L)$ for some
$x \in \bigcup_{k\in D}\mathcal{A}_{k}$. Upon intersection with $G$, their
union forms the set $\mathcal{O}$; that is, writing
${\mathcal{L}}_{D}^{\mathsf{b}} = \sum_{i} \delta _{\gamma _{i}}$, we define
%
\begin{equation}
\label{eq:pf-main1-2} \mathcal{O}\stackrel{\text{def.}} {=} G \cap \widetilde{
\mathcal{I}}_{D}^{
\mathsf{b}} \quad \text{where }
\widetilde{\mathcal{I}}_{D}^{\mathsf{b}}= \bigcup
_{i} \operatorname{range}(\gamma _{i}).
\end{equation}
Recall the notion of a good obstacle set from above Lemma~\ref{lem:hittingobstacle}. We first isolate the following result.
%
\begin{lem}
\label{L:fpp-ann}
For $\ell ,L \geq 2$,
$ L^{-c}\leq \delta \leq \Cr{C:N-bad-1}(\log \ell )^{-
\frac{\nu}{\alpha}} $, $D\subset \{1,\dots ,\ell \}$ with
$|D| \geq \frac{\ell}{2}$, $R$ as in \eqref{eq:R-form} and
$\kappa =\delta L^{\nu} (\log L)^{-b_{2}}$, letting
$\mathbf{G}=\{
\text{$\mathcal{O}$ is a $ (L,R, \frac{|D|}{2}, \kappa  )$-good obstacle set}
\}$, one has
%
\begin{equation}
\label{eq:fpp-ann-1} \mathbb{Q}(\mathbf{G}) \geq 1- \exp \bigl\{- c \delta
^{-
\frac{\alpha}{\nu}} \llvert D \rrvert \bigr\}. 
\end{equation}
\end{lem}
We delay the proof of Lemma~\ref{L:fpp-ann} for a few lines. Since the
point measures ${\mathcal{L}}_{k}^{\mathsf{b}}$ have disjoint supports as
$k$ varies, owing to \eqref{eq:LS-D}, \eqref{eq:AA_i}, and
\eqref{eq:A_i-prop}, the event $\mathbf{B}_{D}$ can be recast in view of
\eqref{eq:B-i} and \eqref{eq:LS-D-really} as
%
\begin{equation}
\label{eq:B_D-rewrite}
\mathbf{B}_{D}=  \bigl\{\mathcal{C} \cap B_{R}^{\mathsf{c}} \neq
\emptyset ,   \mathcal{C} \cap \widetilde{\mathcal{I}}_{D}^{\mathsf{b}} =
\emptyset   \bigr\} \subset  \bigl\{ 0\leftrightarrow B_{R}^{\mathsf{c}}
\text{ in } \mathcal{L}_{\widetilde{\mathcal{G}} \setminus
\widetilde{\mathcal{I}}_{D}^{\mathsf{b}}}  \bigr\}
\end{equation}
(recall the notation from below \eqref{eq:2-point}, as well as from above \eqref{eq:LS-rest} that $\mathcal{L}_{U}$ denotes the
restriction of $\mathcal{L}$ to loops contained in $U$). One issue with
the event $\mathbf{B}_{D}$ rendered visible by \eqref{eq:B_D-rewrite} is
that $\mathcal{L}$ is not independent of
$\widetilde{\mathcal{I}}_{D}^{\mathsf{b}}$---in particular,
\eqref{eq:LS-rest} is not directly applicable in this context with
$U = \widetilde{\mathcal{G}} \setminus \widetilde{\mathcal{I}}_{D}^{
\mathsf{b}}$. To deal with this issue, we proceed as follows. Let
$\widecheck{\mathcal{L}}_{D}^{\mathsf{b}} \stackrel{\text{law}}{=}
\mathcal{L}_{D}^{\mathsf{b}}$ denote a copy independent of $\mathcal{L}$ defined
under $\mathbb{Q}$ by suitably enlarging the underlying probability space.
Let $\mathcal{L}' = \mathcal{L}- \mathcal{L}_{D}^{\mathsf{b}}$, and define
\begin{equation*}
\widecheck{\mathcal{L}}= \widecheck{\mathcal{L}}_{D}^{\mathsf{b}}
+ \mathcal{L}'.
\end{equation*}
Since $\mathcal{L}'$ is independent of $\mathcal{L}_{D}^{\mathsf{b}}$ and
using the definition of $\widecheck{\mathcal{L}}_{D}^{\mathsf{b}}$, one infers
that $\widecheck{\mathcal{L}} \stackrel{\text{law}}{=} \mathcal{L}$ and
that $\widecheck{\mathcal{L}}$ and $\mathcal{L}_{D}^{\mathsf{b}}$ are independent
under $\mathbb{Q}$. In particular, since
$\widetilde{\mathcal{I}}_{D}^{\mathsf{b}}$ is plainly
$\mathcal{L}_{D}^{\mathsf{b}}$-measurable in view of \eqref{eq:pf-main1-2},
one obtains that
%
\begin{equation}
\label{eq:pf-main1-3} \text{$\widecheck{\mathcal{L}}$ and
$\widetilde{\mathcal{I}}_{D}^{\mathsf{b}}$ are independent under
$\mathbb{Q}$.} 
\end{equation}
Moreover,
$\mathcal{L}_{\widetilde{\mathcal{G}} \setminus
\widetilde{\mathcal{I}}_{D}^{\mathsf{b}}} \leq \mathcal{L}'$ by definition of
$\widetilde{\mathcal{I}}_{D}^{\mathsf{b}}$ and, therefore,
$\mathcal{L}_{\widetilde{\mathcal{G}} \setminus
\widetilde{\mathcal{I}}_{D}^{\mathsf{b}}} \le \widecheck{\mathcal{L}}_{
\widetilde{\mathcal{G}} \setminus \widetilde{\mathcal{I}}_{D}^{\mathsf{b}}}$
by monotonicity. Thus, returning to \eqref{eq:B_D-rewrite}, observing that
the event $\mathbf{G}$, defined above \eqref{eq:fpp-ann-1}, is
$\widetilde{\mathcal{I}}_{D}^{\mathsf{b}}$-measurable and that, conditionally
on $\widetilde{\mathcal{I}}_{D}^{\mathsf{b}}$, the loop soup
$\widecheck{\mathcal{L}}_{\widetilde{\mathcal{G}} \setminus
\widetilde{\mathcal{I}}_{D}^{\mathsf{b}}}$ has law
$\mathbb{Q}_{\widetilde{\mathcal{G}} \setminus
\widetilde{\mathcal{I}}_{D}^{\mathsf{b}}}$ owing to \eqref{eq:pf-main1-3} and
the restriction property \eqref{eq:LS-rest}, it follows that
%
\begin{equation}
\label{eq:pf-main1-4}
\mathbb{Q}(\mathbf{B}_{D},   \mathbf{G})\leq \mathbb{Q} \bigl( 0
\leftrightarrow B_{R}^{\mathsf{c}} \text{ in } \widecheck{\mathcal{L}}_{
\widetilde{\mathcal{G}} \setminus \widetilde{\mathcal{I}}_{D}^{\mathsf{b}}}
,   \mathbf{G} \bigr) =2 \mathbb{E}^{\mathbb{Q}} \bigl[ \mathbb{P}_{
\widetilde{\mathcal{G}} \setminus \widetilde{\mathcal{I}}_{D}^{\mathsf{b}}}
\bigl(0\leftrightarrow B_{R}^{\mathsf{c}}\bigr) \cdot 1_{\mathbf{G}}  \bigr],
\end{equation}
where the last step also uses the formula \eqref{eq:pf-Th1-1} applied to
$\widetilde{\mathcal{G}} \setminus \widetilde{\mathcal{I}}_{D}^{\mathsf{b}}$
instead of $\widetilde{\mathcal{G}}$. Applying a union bound over points
on the boundary of $G\setminus B_{R}$ connected to $0$ on the right-hand
side of \eqref{eq:pf-main1-4} and, subsequently, applying Corollary~\ref{cor:2pointkilled},
which is in force on the event $\mathbf{G}$, with the choices
$n=(1-\rho ) \frac{\ell}{2}$ and $\kappa $ as in Lemma~\ref{L:fpp-ann},
one deduces that
\begin{equation*}
\mathbb{Q}(\mathbf{B}_{D}, \mathbf{G})\leq C (\ell
L)^{\alpha } \exp \biggl\{ - \frac{ c(1-\rho )\delta \ell }{(\log L)^{b_{2}}} \biggr\}.
\end{equation*}
Together with the bound on $\mathbb{Q}( \mathbf{G}^{\mathsf{c}})$ provided
by Lemma~\ref{L:fpp-ann}, the upper bound \eqref{eq:pf-main1-1} immediately
follows, which completes the proof.
\end{proof}

It remains to provide the proof of Lemma~\ref{L:fpp-ann}.

\begin{proof}[Proof of Lemma~\ref{L:fpp-ann}]
The event $\mathbf{G}$ in question refers to the existence of a set
$A\subset \operatorname{range}(\pi \cap B_{R})$ with $|A|\geq \frac{|D|}{2}$ for
every path $\pi $ in $\Lambda (L)$ from $0$ to $B_{R}^{\mathsf{c}}$ with
the property that
%
\begin{equation}
\label{eq:pf-ann-1} \operatorname{cap} \bigl(\widetilde{\mathcal{I}}_{D}^{\mathsf{b}}
\cap B(y,L) \bigr) \geq \delta L^{\nu} (\log L)^{-b_{2}},
\quad y \in A. 
\end{equation}
Let $\pi $ be any such path. By definition of $\mathcal{A}_{k}$ in
\eqref{eq:A_i}, the range of $\pi $ intersects $\mathcal{A}_{k}$ for any
$k \in D$. We pick one such point of intersection for every $k\in D$, which
defines the set $A'$. In view of \eqref{eq:defLambda} and
\eqref{eq:R-form}, applying a union bound over the possible choices for
$A'$ yields a factor bounded by
%
\begin{equation}
\label{eq:pf-ann-entro} \bigl\llvert B_{R} \cap \Lambda (L) \bigr\rrvert
^{ \llvert D \rrvert } \leq e^{C \log (\ell ) \llvert D \rrvert }. 
\end{equation}
As we now explain, it is enough to consider a fixed set
$A'=\{x_{k} : k \in D\}$ with $x_{k} \in \mathcal{A}_{k}$ for all
$k\in D$ and to show for
%
\begin{equation}
\label{eq:pf-ann-2} A\stackrel{\text{def.}} {=} \bigl\{x_{k} :
\widetilde{B}(x_{k},L) \text{ contains a big loop in $\mathcal{L}$} \bigr\} 
\end{equation}
that, for $\delta \geq L^{-c}$,
%
\begin{equation}
\label{eq:pf-ann-3} \mathbb{Q} \biggl( \llvert A \rrvert < \frac{ \llvert D \rrvert }{2}
\biggr) \leq e^{-c\delta ^{-
\frac{\alpha}{\nu}} \llvert D \rrvert }. 
\end{equation}
Indeed, the set $A$, defined by \eqref{eq:pf-ann-2}, satisfies
\eqref{eq:pf-ann-1} by definition of
$\widetilde{\mathcal{I}}_{D}^{\mathsf{b}}$ and of big loops (see
\eqref{eq:small-loop}), and $A$ has the required cardinality on the complement
of the event in \eqref{eq:pf-ann-3}. Furthermore, for
$\delta \leq \Cr{C:N-bad-1}(\log \ell )^{-\frac{\nu}{\alpha}}$, the entropy
term \eqref{eq:pf-ann-entro} gets absorbed in the exponential from
\eqref{eq:pf-ann-3}.

To show \eqref{eq:pf-ann-3}, one observes that the events
$E_{k}=\{ \widetilde{B}(x_{k},L)
\text{ contains a big loop in $\mathcal{L}$}\}$ are independent as $k$ varies
since the sets $ \widetilde{B}(x_{k},L)$ are disjoint by construction;
see \eqref{eq:A_i-prop}. Hence, $|A|$ dominates a binomial random variable
with parameters $|D|$ and $p= \inf_{k} \mathbb{Q}(E_{k})$. We will now
argue that, for all $\delta \in (0,1)$ and $\delta \geq L^{-c}$,
%
\begin{equation}
\label{eq:pf-ann-4} p \geq 1 -e^{-c\delta ^{- \frac{\alpha}{\nu}}}. 
\end{equation}
Once this is shown, a union bound gives that
$\mathbb{Q}(|A|<n )\leq 2^{n} (1-p)^{n}$ for $n=|D|/2$ and
\eqref{eq:pf-ann-3} follows for $\delta \in{(0,c)}$ for some small enough
constant $c>0$.

To argue that \eqref{eq:pf-ann-4} holds, we use Lemma~\ref{L:loop-intensities}.
Let $x \in G$. By \eqref{eq:mu-LB} applied with $\zeta =C$ to
$K=B(x, \Cl{c:loop-int2} \delta ^{1/\nu} L)$ for suitable choice of
$\Cr{c:loop-int2}$, we can ensure that
%
\begin{equation}
\label{eq:pf-ann-5} \mathbb{Q}\begin{pmatrix}
\text{$\exists \gamma \in \text{supp}(\mathcal{L})$ s.t.~$\operatorname{range}(
\gamma ) \cap B\bigl(x, \Cr{c:loop-int2} \delta ^{1/\nu} L\bigr) \neq
\emptyset $}
\\
\text{and
$\text{cap}(\gamma ) \geq \delta L^{\nu }\bigl(\log \bigl(\delta ^{1/\nu} L\bigr)\bigr)^{-b_{2}}$
} \end{pmatrix} \geq \Cl[c]{c:loop-int3} 
\end{equation}
whenever $\delta ^{1/\nu} L \geq 1$. Requiring that
$\delta \geq L^{-c}$, one further ensures that
$\log (\delta ^{1/\nu} L) \geq c \log L$ (along with
$\delta ^{1/\nu} L \geq 1$), which effectively allows to remove the factor
$\delta ^{1/\nu}$ from the logarithm in \eqref{eq:pf-ann-5}. Now, applying
\eqref{eq:mu-UB} with $\zeta $ a large enough constant, we can further
ensure that with probability at most $\frac{\Cr{c:loop-int3}}{2}$, a loop
in $\text{supp}(\mathcal{L})$ will intersect both $K$ and the complement
of $B(x, \zeta \Cr{c:loop-int2} \delta ^{1/\nu} L)$. Combining this with
\eqref{eq:pf-ann-5} yields that
%
\begin{equation}
\label{eq:pf-ann-6} \mathbb{Q}\begin{pmatrix}
\text{$\exists \gamma \in \text{supp}(\mathcal{L})$ s.t.~$\operatorname{range}(
\gamma ) \cap B\bigl(x, \Cr{c:loop-int2} \delta ^{1/\nu} L\bigr) \neq
\emptyset $,}
\\
\text{$\operatorname{range}(\gamma ) \subset B\bigl(x, \zeta \Cr{c:loop-int2}
\delta ^{1/\nu} L\bigr)$ and
$\text{cap}(\gamma ) \geq \delta L^{\nu }(\log L)^{-b_{2}}$ } \end{pmatrix} \geq \frac{\Cr{c:loop-int3}}{2}. 
\end{equation}
One now considers, for a given $x_{k}$ as above, the set
$\widetilde{\Lambda }_{k} = B(x_{k}, L) \cap \Lambda (2\zeta
\Cr{c:loop-int2} \delta ^{1/\nu} L)$ so that the boxes
$B(x, \zeta \Cr{c:loop-int2} \delta ^{1/\nu} L)$ are disjoint as $x$ varies
in $\widetilde{\Lambda }_{k}$ by \eqref{eq:defLambda}, and forms the subset
$\Lambda _{k} \subset \widetilde{\Lambda }_{k}$ by keeping only those points
$x$ such that
$B(x, \zeta \Cr{c:loop-int2} \delta ^{1/\nu} L) \subset B(x_{k},L)$. By
\eqref{eq:defLambda} and \eqref{eq:intro_sizeball},
$|\Lambda _{k}| \geq c \delta ^{-\frac{\alpha}{\nu}}$, and the events in
\eqref{eq:pf-ann-6} are independent, as $x$ varies in $\Lambda _{k}$ by
construction. The claim \eqref{eq:pf-ann-4} now follows, since the occurrence
of at least one of these events already implies $E_{k}$.
\end{proof}

\begin{Rk}
\label{rk:simplerproof}
If one is only interested in proving \eqref{eq:def_rho}, that is, obtaining
\eqref{eq:1-arm-UB-crit} for some subpolynomial function $q$, the above
proof of Theorem~\ref{T:1arm-crit} can be simplified as follows. One replaces
the events $\mathbf{B}_{k}$ from \eqref{eq:B-i} by a single event,
\begin{equation*}
\mathbf{B}= \bigl\{\mathcal{C}\cap B_{R}^{\mathsf{c}}\neq
\emptyset , \mathcal{C}\cap B_{R}\text{ does not contain any big loop} \bigr\},
\end{equation*}
where we recall the definition of big loops from below
\eqref{eq:small-loop}, and notice that by a similar reasoning as in
\eqref{eq:pf-Th1-3}
\begin{equation*}
\mathbb{P}\bigl(0\leftrightarrow B_{R}^{\mathsf{c}}\bigr)\leq \mathbb{P}(
\mathbf{B})+\mathbb{P}\bigl(\mathcal{C}\geq \delta L^{\nu}(\log L)^{-b_{2}}\bigr)
\leq \mathbb{P}(\mathbf{B})+C\delta ^{-1/2}L^{-\nu /2}(\log L)^{b_{2}/2}.
\end{equation*}
One can then define $\mathcal{O}$ as the set of vertices in $B_{R}$ hit
by a big loop, show similarly as in Lemma~\ref{L:fpp-ann} that
$\mathcal{O}$ is a $(L,R,\ell /2,\kappa )$-good obstacle with high probability,
and hence proceeding similarly as in \eqref{eq:pf-main1-4} and below one
deduces that
$\mathbb{P}(\mathbf{B})\leq (L\ell )^{\alpha}\exp \{-c\delta (\log L)^{-b_{2}}
\ell \}$. Taking $\delta =c$ and $\ell =(\log L)^{1+b_{2}}$, it follows
that
\begin{equation*}
\mathbb{P}\bigl(0\leftrightarrow B_{R}^{\mathsf{c}}\bigr)\leq CL^{-\nu /2}(
\log L)^{b_{2}/2}\leq CR^{-\nu /2}(\log R)^{(\nu +b_{2}(\nu +1))/2}
\end{equation*}
by \eqref{eq:R-form}. When $\nu >1$, this simply corresponds to a higher
$\log $ power in \eqref{eq:const-UB-1arm-crit}, but when $\nu =1$, for
instance on $\mathbb{Z}^{3}$, one obtains $q(R)=(\log R)^{1/2}$ instead
of $\log \log R$ in \eqref{eq:const-UB-1arm-crit}. Incidentally, this factor
$(\log R)^{1/2}$ on $\mathbb{Z}^{3}$ already appeared multiple times before
(see \cite{DiWi,DrePreRod5}), and it can also be deduced from
\cite{DrePreRod3}, (1.9), hence the interest of improving it to
$\log \log R$. Intuitively, this improvement comes from asking to meet
many big loops in Proposition~\ref{P:N-bad}, whose union have a much bigger
capacity than a single loop by Lemma~\ref{L:cap-LB}.
\end{Rk}

\section{Bounds on the two-point function}
\label{sec:2point}

In this section we prove Theorem~\ref{thm:2pointUB} and Corollary~\ref{cor:2pointZ3}
and, at the very end, provide a short derivation of
\eqref{eq:lowerboundtau}. One important tool will be the random interlacements
model, originally introduced on $\mathbb{Z}^{d}$ in \cite{MR2680403} (cf. also
\cite{DrRaSa-13} for an introduction to the model), extended to general
transient graphs in \cite{MR2525105}, and to the cable system in
\cite{MR3502602}. We refer, for instance, to
\cite{DrePreRod3}, Section~2.5, for a brief introduction of the model at
the level of generality needed in the present context.

Under the probability $\overline{\mathbb{P}}$, we denote by
$\widetilde{{\mathcal{I}}}^{u}$ the closure of the interlacement set at level
$u>0$ on $\widetilde{\mathcal{G}}$ and by
$\widetilde{{\mathcal{V}}}^{u}\stackrel{\text{def.}}{=}
\widetilde{\mathcal{G}}\setminus \widetilde{{\mathcal{I}}}^{u}$ the corresponding
vacant set on $\widetilde{\mathcal{G}}$. We further abbreviate by
${\mathcal{I}}^{u}\stackrel{\text{def.}}{=}\widetilde{{\mathcal{I}}}^{u}
\cap G$ the (usual) interlacement set on $G$. With the notation introduced
above Corollary~\ref{cor:2pointkilled}, conditionally on
$\widetilde{{\mathcal{I}}}^{u}$, the measure
$\mathbb{P}_{\widetilde{{\mathcal{V}}}^{u}}$ is the law of a Gaussian free field
killed on $\widetilde{{\mathcal{I}}}^{u}$. The interest of random interlacements
in the study of the Gaussian free field is due to an isomorphism theorem
between the two objects, first observed in \cite{MR2892408} and then improved
in \cite{MR3502602,MR3492939,DrePreRod3}. Our first result is a useful
consequence of this isomorphism, combined with the restriction property
\eqref{eq:LS-rest} and the loop soup isomorphism, in the (weak) form
\eqref{eq:iso-LS}. For $a \in \mathbb{R}$, let
$ \widetilde{\mathcal{K}}^{a}= \widetilde{\mathcal{K}}^{a}(\varphi )
\subset \widetilde{\mathcal{G}}$ denote the connected component of
$0$ (see \eqref{eq:0}) in
$\{ x \in \widetilde{\mathcal{G}} : \varphi _{x} \geq a\}$.

\begin{lem}
\label{lem:couplingwithinter}
For each $u>0$, there exists a coupling of
$(\varphi ,\widetilde{{\mathcal{V}}}^{u},\gamma )$, with $\varphi $ having law
$\mathbb{P}$, $\widetilde{{\mathcal{V}}}^{u}$ a vacant set with the same law
as under $\overline{\mathbb{P}}$, and $\gamma $ with law
$\mathbb{P}_{{\widetilde{{\mathcal{V}}}^{u}}}$, conditionally on
$\widetilde{{\mathcal{V}}}^{u}$, such that
\begin{equation*}
\widetilde{\mathcal{K}}^{\sqrt{2u}}\subset \bigl\{x\in{ \widetilde{{
\mathcal{V}}}^{u}}: \gamma _{x}\geq 0 \bigr\}.
\end{equation*}
\end{lem}
\begin{proof}
By \cite{DrePreRod3}, Theorem~1.1 and~Lemma~3.4(2), the isomorphism (Isom)
on page 259 therein is satisfied for any graph satisfying
\eqref{eq:intro_Green}, such as $\mathcal{G}$. Using the symmetry of the
Gaussian free field, one readily deduces from this isomorphism that
%
\begin{equation}
\label{eq:iso-applied} \widetilde{\mathcal{K}}^{\sqrt{2u}} \text{ has the same law under $\mathbb{P}$ as }
\widetilde{\mathcal{K}}^{0}1 \bigl\{\widetilde{\mathcal{K}}^{0}
\subset{ \widetilde{{\mathcal{V}}}^{u}} \bigr\}\text{ under }
\mathbb{P}\otimes \overline{\mathbb{P}}, 
\end{equation}
where $A1\{E\}$ is the set which is equal to $A$ if the event $E$ occurs
and is equal to $\emptyset $ otherwise. By \eqref{eq:iso-LS} and the symmetry
of $\mathbb{P}$ under $\varphi \mapsto -\varphi $, the set
$\widetilde{\mathcal{K}}^{0}$ is either empty with probability
$\frac{1}{2}$ or has the same law as the cluster $\mathcal{C}$ of
$0$ in ${\mathcal{L}}$ under $\mathbb{Q}$ otherwise. Moreover, since
$\widetilde{{\mathcal{V}}}^{u}$ and ${\mathcal{L}}$ are independent, it follows
from \eqref{eq:LS-rest} that, conditionally on
$\widetilde{{\mathcal{V}}}^{u}$, the loops
${\mathcal{L}}_{\widetilde{{\mathcal{V}}}^{u}}$ which are entirely contained
in $\widetilde{{\mathcal{V}}}^{u}$ have the same law as a loop soup under
$\mathbb{Q}_{\widetilde{{\mathcal{V}}}^{u}}$. Therefore, conditionally on
$\widetilde{{\mathcal{V}}}^{u}$, the set
$\widetilde{\mathcal{K}}^{\sqrt{2u}}$ is stochastically dominated by a
set which is either empty with probability $\frac{1}{2}$ or is the cluster
of $0$ in a loop soup under $\mathbb{Q}_{\widetilde{{\mathcal{V}}}^{u}}$ otherwise
(to see this, notice that whenever the latter set is nonempty, the inclusion
from \eqref{eq:iso-applied} translates into the requirement that the cluster
of $0$ in ${\mathcal{L}}_{\widetilde{{\mathcal{V}}}^{u}}$ is equal to the cluster
of $0$ in ${\mathcal{L}}$). Using \eqref{eq:iso-LS} but for the graph
$\widetilde{\mathcal{G}}$ with infinite killing on
${\widetilde{{\mathcal{I}}}^{u}}$ and the symmetry of the Gaussian free field
again, the claim follows.
\end{proof}

As we detail below, one obtains from the first passage percolation upper
bound in \cite{prevost2023passage}, Theorem~5.4, that the set
${\mathcal{I}}^{u}$ is a $(L,R,n,\kappa )$-good obstacle set (see above Lemma~\ref{lem:hittingobstacle})
with high probability for an appropriate choice of $L$, $n$, $\kappa $ and combining
this with Corollary~\ref{cor:2pointkilled} and Lemma~\ref{lem:couplingwithinter},
we now deduce Theorem~\ref{thm:2pointUB}.

\begin{proof}[Proof of Theorem~\ref{thm:2pointUB}]
The statement is trivial for $a=0$ and by
\cite{DrePreRod3}, Lemma~4.3, we may assume without loss of generality
that $a>0$. Abbreviate $u=a^{2}/2$, let $\eta \in{(0,1)}$, and recall
$\tau _{a}^{\mathrm{tr}}$ from \eqref{eq:2-point}. For $a>0$, the truncation
$\{x\nleftrightarrow \infty \text{ in }E^{\geq a}\}$ has probability
one; hence, applying Lemma~\ref{lem:couplingwithinter}, we have that
%
\begin{equation}
\label{eq:bound2pointviacoupling}
\tau _{a}^{\mathrm{tr}}(0,x)= \tau _{\sqrt{2u}}^{\mathrm{tr}}(0,x)\leq
\overline{\mathbb{E}}  \bigl[\mathbb{P}_{\widetilde{{\mathcal{V}}}^{u}}
\bigl(0\leftrightarrow x\text{ in }\bigl\{y\in{\widetilde{{\mathcal{V}}}^{u}}:
\varphi _{y}\geq 0\bigr\} \bigr)  \bigr].
\end{equation}
To bound the right-hand side of \eqref{eq:bound2pointviacoupling}, we use
Corollary~\ref{cor:2pointkilled}. Recall the notion of a good obstacle
set introduced above Lemma~\ref{lem:hittingobstacle}. By
\cite{prevost2023passage}, (6.4) (see also below
\cite{prevost2023passage}, (2.22), regarding the function $F_{\nu}$ appearing
therein), which holds without the condition
\cite{prevost2023passage}, (2.5), for similar reasons as explained in
\cite{prevost2023passage}, Remark~6.6, there exist
$c,\Cl[c]{ctwopoint2}>0$ and $C<\infty $, depending on $\eta $, such that
letting $L=Cu^{-\frac{1}{\nu}}$, $n=cRu^{\frac{1}{\nu}}$, and
$\kappa =u^{-1}(\log (\frac {1}{u})\vee 1)^{-b_{2}}$, if
$u^{\frac{1}{\nu}}R\geq C$, then
%
\begin{equation}
\label{eq:boundM} \overline{\mathbb{P}} \bigl({\mathcal{I}}^{u}\text{ is a }
(L,R(1-\eta ),n, \kappa )\text{-good obstacle} \bigr)\geq 1-\exp \biggl\{-
\frac{\Cr{ctwopoint2}u^{\frac{1}{\nu}}R(1-\eta )}{(\log (u^{\frac{1}{\nu}}R(1-\eta ))^{b_{1}})} \biggr\}. 
\end{equation}
We now set $R=d(0,x)$ in what follows and first assume that
$u^{\frac{1}{\nu}}d(0,x) \geq C$. We start with observing that:
\begin{itemize}
\item $d(x,B_{R(1-\eta )+CL})\geq (\eta /2)d(x,0)$ if
$Ra^{\frac{2}{\nu}}$ is large enough, and
\item $d(0,x)^{-\nu}\leq C\tau _{0}^{\mathrm{tr}}(0,x)$ due to
\eqref{eq:intro_Green} and \cite{MR3502602}, Proposition~5.2.
\end{itemize}
As a consequence, we can apply Corollary~\ref{cor:2pointkilled} (with
$R(1-\eta )$ in place of $R$ therein) in order to obtain that, on the event
in the probability of \eqref{eq:boundM}, we can upper bound
%
\begin{equation}
\label{eq:goodBd}
\mathbb{P}_{\widetilde{{\mathcal{V}}}^{u}} \bigl(0\leftrightarrow x
\text{ in }\bigl\{y\in{\widetilde{{\mathcal{V}}}^{u}}:\varphi _{y}\geq 0\bigr\} \bigr)
\leq C\tau _{0}^{\mathrm{tr}}(0,x) \exp  \Big\{ -
\frac{c\kappa n}{L^{\nu}}  \Big\}.
\end{equation}
On the complement of that event, we can use the
$\overline{\mathbb{P}}$-a.s. bound
\begin{equation*}
\mathbb{P}_{\widetilde{{\mathcal{V}}}^{u}} \bigl(0\leftrightarrow x
\text{ in }\bigl\{y\in{\widetilde{{\mathcal{V}}}^{u}}:\varphi _{y}\geq 0\bigr\} \bigr)
\leq Cd(0,x)^{-\nu} \le C\tau _{0}^{\mathrm{tr}}(0,x)
\end{equation*}
(due to \cite{MR3502602}, Proposition~5.2, again). As a consequence, plugging
the choices of $u$, $L$, $n$, and $\kappa $ from before \eqref{eq:boundM} into
the right-hand side of \eqref{eq:goodBd} and combining the previous two
displays with \eqref{eq:boundM} for $\eta =1/2$, one obtains
\eqref{eq:bound2point} when $a^{\frac{2}{\nu}}R\geq C$. Finally, when
$a^{\frac{2}{\nu}}R\leq C$, one can easily check that
\eqref{eq:bound2point} remains valid by monotonicity of
$a\mapsto \tau _{a}^{\mathrm{tr}}(0,x)$ modulo adapting the constant
$\Cr{Ctwopoint}$.
\end{proof}

We now turn to the proof of Corollary~\ref{cor:2pointZ3}, which borrows
techniques from \cite{DrePreRod5}. To ease its reading, we first provide
an overview of the argument, wilfully ignoring certain technical points.
The upper bound in \eqref{eq:2pointZ3} comes as slight strengthening (needed
to obtain the pre-factor $\frac{\pi}{3}$) of the upper bound in Theorem~\ref{thm:2pointUB},
when specializing to the case of $\mathbb{Z}^{3}$. The lower bound in
\eqref{eq:2pointZ3} relies on a change of measure technique (similar to
\cite{DrePreRod5}), which roughly works as follows: to bound
$\tau ^{\mathrm{tr}}_{a}(0,[\lambda \xi e])$ for $a > 0$ from below, one
first forces the clusters of $0$ and $[\lambda \xi e]$ in balls of radius
$\approx \xi $ to have large capacity, the cost of which is commensurate
with $\tau ^{\mathrm{tr}}_{0}(0,[\lambda \xi e])$ (this is not quite true,
cf.~\eqref{eq:proof2point2} below; this forcing is also at the origin of
the $\log \log \xi $ term appearing in \eqref{eq:2pointZ3}). After conditioning
on these two clusters, in the remaining region between the balls, one ``aids''
the field by shifting the measure harmonically, as for the excursion sets
of interest to be $\geq -\eta a$ everywhere, which is slightly super-critical.
Under the shifted measure, the interlacement trajectories present in the
associated random walk picture at the shifted level then produce the remaining
connection with high probability; compare~\eqref{eq:proof2point4} below;
in particular, these trajectories hit the two clusters in question owing
to their large capacity. The entropic cost resulting from the shift gives
rise to the lower bound in \eqref{eq:2pointZ3}.

\begin{proof}[Proof of Corollary~\ref{cor:2pointZ3}]
Throughout this proof we assume that $d(x,y)=|x-y|_{2}$ is the Euclidian
distance on $\mathbb{Z}^{3}$, and similarly as before we can assume without
loss of generality that $a>0$. Furthermore, while $\nu =1$ in the current
setting, we still write $\nu $ for enhanced comparison with the general
setting of \cite{DrePreRod5}, which we frequently refer to in the sequel.
For the upper bound, one notices that the constant $\Cr{ctwopoint2}$ from
\eqref{eq:boundM} can be taken equal to $(1-\eta )\pi /3$ when
$Ru^{\frac{1}{\nu}}$ is large enough since by definition it is equal to
$(1-\eta )$ times the constant $c_{38}=c_{4}$ from
\cite{prevost2023passage}, Theorem~5.4, and that $c_{4}=\pi /3$ on
$\mathbb{Z}^{3}$ by \cite{prevost2023passage}, (2.6). Noting that when
$\lambda $ is large enough, then \eqref{eq:boundM} with the choices
$R=\lambda \xi =\lambda a^{-\frac{2}{\nu}}$ and $u=a^{2}/2$ constitutes
the main contribution to the bound \eqref{eq:bound2point} with
$R=d(x,y)$, one readily concludes the upper bound after a change of variable
for $\eta $.

For the lower bound, we follow a strategy similar to the proof of
\cite{DrePreRod5}, (8.3), which, however, requires an adaptation to obtain
the exact constant $\pi /6$ in \eqref{eq:2pointZ3} that we are now going
to detail. Let $a>0$, $R=\lambda \xi $,
${h}_{K\cup K'}(x)=P_{x}(H_{K\cup K'}<\infty )$ for
$K,K'\subset \widetilde{\mathcal{G}}$ and
\begin{equation*}
A\bigl(K,K',a\bigr)=  \bigl\{K\leftrightarrow K' \text{ in } \bigl\{x\in{
\widetilde{\mathcal{G}}\setminus \bigl(K\cup K'\bigr)}: \varphi _{x}\geq a\bigl(1-
\eta -{h}_{K\cup K'}(x)\bigr) \bigr\}  \bigr\},
\end{equation*}
where for $K$ and $K'$ connected, $K\leftrightarrow K'$ in $A$ means that
$K\cup K'\cup A$ is a connected set in $\widetilde{\mathcal{G}}$. Then
similarly as in \cite{DrePreRod5}, (6.17), it follows from the Markov property
of the Gaussian free field that, for all $s,t>0$,
%
\begin{equation}
\label{eq:proof2point1} %
\begin{aligned} &\tau ^{\mathrm{tr}}_{(1-\eta )a}
\bigl(0,[Re] \bigr)
\\
&\quad \geq \mathbb{E} \bigl[1 \bigl\{\operatorname{cap} \bigl( \widetilde{
\mathcal{K}}_{s\xi}^{a}(0) \bigr)\geq t\xi
^{\nu},\operatorname{cap} \bigl( \widetilde{\mathcal{K}}_{s\xi}^{a}(R)
\bigr)\geq t\xi ^{\nu} \bigr\} \mathbb{P}_{\mathcal{U}} \bigl(A
\bigl(\widetilde{\mathcal{K}}_{s\xi}^{a}(0), \widetilde{
\mathcal{K}}_{s\xi}^{a}(R),a \bigr) \bigr) \bigr],
\end{aligned} %
\end{equation}
where $\widetilde{\mathcal{K}}^{a}_{s\xi}(0)$, respectively,
$\widetilde{\mathcal{K}}_{s\xi}^{a}(R)$, denotes the cluster of $0$, respectively,
$[Re]$, in $\{x\in{B(0,{s\xi})}:\varphi _{x}\geq a\}$, respectively,
$\{x\in{B([Re],{s\xi})}:\varphi _{x}\geq a\}$, and
$\mathcal{U}=\widetilde{\mathcal{G}}\setminus (
\widetilde{\mathcal{K}}_{s\xi}^{a}(0)\cup \widetilde{\mathcal{K}}_{s
\xi}^{a}(R))$. One can control the first term on the right-hand side of
\eqref{eq:proof2point1} by the FKG inequality and an adaptation of the
proof of \cite{DrePreRod5}, Lemma~6.2, which implies that there exists
$\Cl[c]{cts}>0$ such that, for all $0\leq a\leq c$,
%
\begin{equation}
\label{eq:proof2point2} \mathbb{P} \bigl(\operatorname{cap} \bigl(\widetilde{
\mathcal{K}}_{s\xi}^{a}(0) \bigr) \geq
\Cr{cts}s^{\nu}\xi ^{\nu},\operatorname{cap} \bigl(
\widetilde{\mathcal{K}}_{s
\xi}^{a}(R) \bigr)\geq
\Cr{cts}s^{\nu}\xi ^{\nu} \bigr)\geq c\xi
^{-\nu}q( \xi )^{-2}\exp \bigl\{-Cq(\xi ) \bigr\},
\end{equation}
for some constants $c=c(s)>0$ and $C=C(s)<\infty $. Let us now bound the
probability on the right-hand side of \eqref{eq:proof2point1} when
$t=\Cr{cts}s^{\nu}$. For $M\geq 1$ to be chosen later, we abbreviate
$\ell =M\xi (\log (R/\xi ))^{(2\nu +1)/\nu}$ as in
\cite{DrePreRod5}, (6.19), write $\Cl[c]{c13}$ for the constant which is
equal to the constant $c_{13}$ from \cite{DrePreRod5}, Theorem~5.1, recall
the definition of the balls
$\widetilde{B}(x,L)\subset \widetilde{\mathcal{G}}$ from below
\eqref{eq:LS-D}, and denote by
%
\begin{equation}
\label{eq:defLell} \mathcal{L}_{\ell}\stackrel{\mathrm{def.}} {=}
\bigcup_{i=0}^{\lceil 2R/
\ell \rceil }\widetilde{B} \bigl(
\bigl[(i\ell /2)e \bigr],\Cr{c13}\ell \bigr), 
\end{equation}
the set which corresponds to the one introduced in
\cite{DrePreRod5}, (6.13) (the set in \eqref{eq:defLell} should not be
confused with the loop soup which plays no role here). Similarly as above
\cite{DrePreRod5}, (8.2), let
\begin{align*}
\mathcal{L}'_{\ell}&\stackrel{\mathrm{def.}} {=}
\mathcal{L}_{\ell} \cup \widetilde{B} \bigl(0,\sigma '
\ell \bigr)\cup \widetilde{B} \bigl([Re],\sigma ' \ell \bigr),
\\
\mathcal{L}''_{\ell}&\stackrel{
\mathrm{def.}} {=}\mathcal{L}'_{\ell} \setminus \bigl(
\widetilde{B}(0,\sigma \xi )\cup \widetilde{B} \bigl([Re], \sigma \xi \bigr)
\bigr),
\end{align*}
where $\sigma '\geq \Cr{c13}$ and $\sigma \geq s$ are constants we will
fix later. For $K\subset \widetilde{B}(0,s\xi )$ and
$K'\subset \widetilde{B}([Re],s\xi )$, abbreviating
$U=\widetilde{\mathcal{G}}\setminus (K\cup K')$, let us also denote by
$\mathbb{P}_{U}^{a,\ell}$ the law of
$(\varphi _{x}+a\overline{h}_{\ell}(x))_{x\in{U}}$ under
$\mathbb{P}_{U}$; see above Corollary~\ref{cor:2pointkilled}, where
$\overline{h}_{\ell}(x)=P_{x}(H_{\mathcal{L}''_{\ell}}<H_{K\cup K'})$.
Then similarly as in \cite{DrePreRod5}, (6.17), by a change of measure,
one has that
%
\begin{equation}
\label{eq:proof2point3} \mathbb{P}_{U} \bigl(A \bigl(K,K',a
\bigr) \bigr)\geq \mathbb{P}_{U}^{a,\ell} \bigl(A
\bigl(K,K',a \bigr) \bigr)\exp \biggl\{- \frac{a^{2}\operatorname{cap}_{U}(\mathcal{L}_{\ell}'')+1/e}{2\mathbb{P}_{U}^{a,\ell}(A(K,K',a))}
\biggr\}, 
\end{equation}
where $\operatorname{cap}_{U}(\mathcal{L}''_{\ell})$ denotes the capacity of
$\mathcal{L}''_{\ell}$ associated to the diffusion $X$ on
$\widetilde{\mathcal{G}}$ killed on hitting $U^{\mathsf{c}}$ (see below
\cite{DrePreRod5}, Corollary~5.2, for a rigorous definition). As we will
soon explain, similarly as in \cite{DrePreRod5}, (8.2), uniformly in
$\sigma \geq s$, one can fix $s=s(\eta )\geq 1$,
$\sigma '\geq \Cr{c13}$ and $M\geq \sigma '$ such that, for all
$K\subset \widetilde{B}(0,{s\xi})$,
$K'\subset \widetilde{B}([Re],{s\xi})$ with
$\operatorname{cap}(K)\geq \Cr{cts}s^{\nu}\xi ^{\nu}$, and
$\operatorname{cap}(K')\geq \Cr{cts}s^{\nu}\xi ^{\nu}$, still abbreviating
$U=\widetilde{\mathcal{G}}\setminus (K\cup K')$, if
$R/\xi \geq C=C(\sigma )$, then
%
\begin{equation}
\label{eq:proof2point4} \mathbb{P}_{U}^{a,\ell} \bigl(A
\bigl(K,K',a \bigr) \bigr)\geq 1-\eta . 
\end{equation}
Moreover, proceeding similarly as in
\cite{DrePreRod5}, (7.7) and~(7.8), one can fix $\sigma =\sigma (s)$ large
enough so that
%
\begin{equation}
\label{eq:proof2point5} \operatorname{cap}_{K\cup K'} \bigl(
\mathcal{L}_{\ell}'' \bigr)\leq
\operatorname{cap} \bigl( \mathcal{L}_{\ell}\cap \mathbb{Z}^{3}
\bigr)+C \bigl(s^{\nu}\xi ^{\nu}+ \bigl(\sigma
' \bigr)^{
\nu}\ell ^{\nu} \bigr)\leq
\frac{\pi}{3}(1+\eta )^{2} \frac{R}{\log (\frac{2R}{\ell})}, 
\end{equation}
where the last equality follows from
\cite{prevost2023passage}, (2.24), with
$n=P=\lceil 2R/\ell \rceil +1$ and $N=R$ therein, and upon taking
$R\geq C\xi $ for some constant $C=C(\eta ,\sigma ',s,M)$. Note that the
constant $C_{13}=C_{4}$ appearing therein is equal to $\pi /3$ on
$\mathbb{Z}^{3}$ by \cite{prevost2023passage}, (2.6), and we refer to
\cite{GRS21}, Lemma~2.2, and \cite{dembo2022capacity}, Lemma~2.1, for similar
statements proved directly on $\mathbb{Z}^{3}$. Combining
\eqref{eq:const-UB-1arm-crit}, \eqref{eq:proof2point1},
\eqref{eq:proof2point2}, \eqref{eq:proof2point3},
\eqref{eq:proof2point4}, and \eqref{eq:proof2point5} and noting that
$\log (\frac{2R}{\ell})\geq (1-\eta )\log (\lambda )$ if $\lambda $ is
large enough by our choices of $R$ and $\ell $, one easily deduces the
lower bound in \eqref{eq:2pointZ3} after a change of variable in
$\eta $ and $a$.

It remains to prove \eqref{eq:proof2point4}. By definition of
$\mathbb{P}_{U}^{a,\ell}$ and $A(K,K',a)$, we have
%
\begin{equation}
\label{eq:proof2point6}
\begin{aligned} \mathbb{P}_{U}^{a,\ell} \bigl(A\bigl(K,K',a\bigr) \bigr)&\geq
\mathbb{P}_{U} \bigl(K\leftrightarrow K' \text{ in } \bigl\{x\in{U}:
\varphi _{x}\geq a\bigl(1-\eta -{h}_{K\cup K'}(x)-\overline{h}_{\ell}(x)\bigr)
\bigr\} \bigr)
\\
&\geq \mathbb{P}_{U}  \bigl(K\leftrightarrow K' \text{ in }\bigl\{x\in{
\mathcal{L}'_{\ell}\cap U}: \varphi _{x}\geq -\eta a\bigr\}  \bigr)
\end{aligned}
%
\end{equation}
since
$h_{K\cup K'}(x)+\overline{h}_{\ell}(x)=P_{x}(H_{K\cup K'}<\infty )+P_{x}(H_{
\mathcal{L}_{\ell}''}<H_{K\cup K'})\geq 1$ for all
$x\in{\mathcal{L}'_{\ell}}$ by definition as long as
$\sigma '\ell >\sigma \xi $, which holds as long as $R/\xi $ is large enough,
depending on $\sigma $, $\sigma '$. For $\sigma '$ and
$M\geq \sigma '$ large enough, depending only on $s$ and $\eta $, one can
bound from below the last probability in \eqref{eq:proof2point6} by
\begin{equation*}
(1-\eta /2) \bigl(1-\exp \bigl\{-Ca^{2}\operatorname{cap}(K) \bigr
\} \bigr) \bigl(1-\exp \bigl\{-Ca^{2} \operatorname{cap}
\bigl(K' \bigr) \bigr\} \bigr)
\end{equation*}
for some constants $c=c(\eta )>0$ and $C=C(\eta )<\infty $ using a reasoning
similar to the proof of \cite{DrePreRod5}, (8.2). Indeed, the proof essentially
consists of defining three independent random interlacements, each at level
$(\eta a)^{2}/6$, whose union is included in
$\{\varphi \geq -\eta a\}$ via the isomorphism theorem, such that the following
events occur: the first interlacement has a trajectory hitting $K$ and
going to $\infty $, the second interlacement has a trajectory hitting
$K'$ and going to $\infty $, and the last interlacement has a connected
component in $\mathcal{L}_{\ell}'\cap U$ which intersects both the first
and second interlacement. The probability of the intersection of the two
first events can be bounded from below by
$(1-\exp \{-Ca^{2}\operatorname{cap}(K)\})(1-\exp \{-Ca^{2}\operatorname{cap}(K')
\})$ owing to \cite{DrePreRod5}, Lemma~7.2, and the probability of the
last one can be lower bounded by $1-\eta /2$ upon taking $\sigma '$ and
$M$ large enough; we refer to \cite{DrePreRod5}, (7.17), (7.20), for as
to why and leave the details to the reader. The inequality
\eqref{eq:proof2point4} then follows readily when
$\operatorname{cap}(K)\geq \Cr{cts}s^{\nu}\xi ^{\nu}$,
$\operatorname{cap}(K')\geq \Cr{cts}s^{\nu}\xi ^{\nu}$, and $s=s(\eta )$ is a
sufficiently large constant.
\end{proof}

\begin{Rk}
\label{rem5.2}
As should be clear from \eqref{eq:proof2point3},
\eqref{eq:proof2point5}, and \eqref{eq:proof2point6} (see also
\cite{prevost2023passage}, Proposition~3.2, for the upper bound), the intuitive
reason one is able to obtain the exact constant $\pi /6$ in Corollary~\ref{cor:2pointZ3}
is that
%
\begin{equation}
\label{eq:equivtwopoint} \log \biggl( \frac{\tau ^{\mathrm{tr}}_{a}(0,[R e])}{\tau ^{\mathrm{tr}}_{0}(0,[R e])} \biggr)\sim -
\frac{a^{2}\operatorname{cap}(\mathcal{L}_{\ell})}{2} 
\end{equation}
as $R\rightarrow \infty $, where $\mathcal{L}_{\ell}$ is a ``tube'' (see
\eqref{eq:defLell}) from $0$ to $Re$ of length $R$ and width $\ell $, and
that by \cite{prevost2023passage}, Lemma~2.1 (see also
\cite{dembo2022capacity}, Lemma~2.1), for a similar statement directly
on $\mathbb{Z}^{3}$,
\begin{equation*}
\operatorname{cap}(\mathcal{L}_{\ell})\sim \frac{\pi R}{3\log (R/\ell )} \sim
\frac{\pi R}{3\log (R/\xi )}
\end{equation*}
as long as $\ell $ is roughly of order $\xi $. One could similarly obtain
an explicit constant for graphs satisfying \eqref{eq:intro_sizeball} and
\eqref{eq:intro_Green} with $\nu \leq 1$ as long as the function
$g(x,y)d(x,y)^{-\nu}$ converges as $d(x,y)\rightarrow \infty $; see
\cite{prevost2023passage}, (2.2) and~Lemma~2.1. However, when
$\nu >1$, for instance, on $\mathbb{Z}^{\alpha}$, $\alpha >3$, even if
\eqref{eq:equivtwopoint} was true with $\ell $ of order $\xi $, one would
not obtain the correct constant anymore as
$\operatorname{cap}(\mathcal{L}_{\ell})\asymp R\ell ^{\nu -1}$ (see
\cite{prevost2023passage}, Lemma~2.1) and thus depends on the exact choice
of the constant $\ell /\xi $, which is a priori not clear. Note that
$\ell $ is actually of order $\xi (\log (r/\xi ))^{C}$ in the proof of
\eqref{eq:lowerboundtau}, hence the additional logarithmic factor therein.
\end{Rk}

We conclude with the following brief:

\begin{proof}[Proof of \eqref{eq:lowerboundtau}]
We start with observing that \cite{DrePreRod5}, (8.3), in combination with
\cite{prevost2023passage}, Remark~6.5,4), entails that, for our choices
of parameters, we have
%
\begin{equation}
\label{eq:lowerboundtauaInvent} \tau _{a/C}^{\mathrm{tr}}(x,y)\geq \xi
^{-\nu}q(\xi )^{-2}\exp \biggl\{-{C}q(\xi )-
\frac{{C}(d(x,y)/\xi )}{(\log (d(x,y)/\xi ))^{b}} \biggr\}, 
\end{equation}
where $q$ is as in Theorem~\ref{T:1arm-crit}, $\xi $ as in
\eqref{eq:defxi}, $b = 1$ if $\nu = 1$, and $b = 1 - \nu $ for
$\nu \in (1,\alpha /2)$. Now, due to \eqref{eq:intro_Green}, we have that
\cite{MR3502602}, Proposition~5.2, entails
$\tau _{0}^{\mathrm{tr}}(x,y) \le Cd(x,y)^{-\nu}$, and we furthermore note
that the term $q(\xi )^{-2}\exp \{- C q(\xi )\}$ appearing in
\eqref{eq:lowerboundtauaInvent} is now negligible by Theorem~\ref{T:1arm-crit},
and our condition on $d(x,y)$; inequality \eqref{eq:lowerboundtau} follows.
\end{proof}

\medskip

{\bf Acknowledgments:}
This research was supported through the programs ``Oberwolfach Research
Fellows'' and ``Oberwolfach Leibniz Fellows'' by the Mathematisches
Forschungs\-institut Oberwolfach in 2023. PFR thanks the Research Institute
for Mathematical Sciences (RIMS), an International Joint Usage/Research
Center located in Kyoto University, for its hospitality.

The research of AD has been supported by the Deutsche Forschungsgemeinschaft
(DFG) Grant DR 1096/2-1.

AP has been supported by the Engineering and Physical
Sciences Research Council (EPSRC) Grant EP/R022615/1, Isaac Newton Trust
(INT) grant G101121, European Research Council (ERC) starting grant 804166
(SPRS), the Swiss NSF, and the Deutsche Forschungsgemeinschaft (DFG)---Projektnummer
552316285.

The research of PFR has been supported by the European Research Council (ERC) under the European Union's Horizon Europe research and innovation programme (grant agreement No 101171046).

\bibliography{bibliographie}
\bibliographystyle{abbrv}

\end{document}